\documentclass[reqno]{amsart}
\usepackage{amsmath, amsthm, amscd, amssymb, amsfonts, amsbsy}
\usepackage{latexsym}
\usepackage{mathrsfs}
\usepackage{color}
\usepackage{enumerate}
\usepackage{pxfonts}
\usepackage{verbatim}

\theoremstyle{plain}
\newtheorem{theorem}[equation]{Theorem}
\newtheorem{lemma}[equation]{Lemma}
\newtheorem{corollary}[equation]{Corollary}
\newtheorem{proposition}[equation]{Proposition}

\theoremstyle{definition}
\newtheorem{definition}[equation]{Definition}
\newtheorem{condition}[equation]{Condition}

\theoremstyle{remark}
\newtheorem{remark}[equation]{Remark}

\newcommand{\dv}{\operatorname{div}}

\newcommand{\dist}{\operatorname{dist}}
\newcommand{\diam}{\operatorname{diam}}
\newcommand{\loc}{\operatorname{loc}}

\newcommand{\tr}{\operatorname{tr}}

\newcommand*{\tran}{^{\mkern-1.5mu\mathsf{T}}}

\numberwithin{equation}{section}

\newcommand{\bN}{\mathbb{N}}

\newcommand{\bR}{\mathbb{R}}
\newcommand{\bZ}{\mathbb{Z}}

\newcommand\cD{\mathcal{D}}

\newcommand\rU{\mathrm{U}}

\providecommand{\set}[1]{\{#1\}}
\providecommand{\Set}[1]{\left\{#1\right\}}

\providecommand{\abs}[1]{\lvert#1\rvert}
\providecommand{\Abs}[1]{\left\lvert#1\right\rvert}

\providecommand{\norm}[1]{\lVert#1\rVert}

\renewcommand{\vec}[1]{\boldsymbol{#1}}

\begin{document}
\title[On oblique derivative problem]
{On conormal and oblique derivative problem for elliptic equations with Dini mean oscillation coefficients}

\author[H. Dong]{Hongjie Dong}
\address[H. Dong]{Division of Applied Mathematics, Brown University,
182 George Street, Providence, RI 02912, United States of America}
\email{Hongjie\_Dong@brown.edu}
\thanks{H. Dong was partially supported by the NSF under agreement DMS-1600593.}

\author[J. Lee]{Jihoon Lee}
\address[J. Lee]{Department of Mathematics, Yonsei University, 50 Yonsei-ro, Seodaemun-gu, Seoul 03722, Republic of Korea}
\email{jjhlee24@gmail.com}

\author[S. Kim]{Seick Kim}
\address[S. Kim]{Department of Mathematics, Yonsei University, 50 Yonsei-ro, Seodaemun-gu, Seoul 03722, Republic of Korea}
\email{kimseick@yonsei.ac.kr}
\thanks{S. Kim is partially supported by NRF Grant No. NRF-2016R1D1A1B03931680 and No. NRF-20151009350}



\begin{abstract}
We show that weak solutions to conormal derivative problem for elliptic equations in divergence form are continuously differentiable up to the boundary provided that the mean oscillations of the leading coefficients satisfy the Dini condition, the lower order coefficients satisfy certain suitable conditions, and the boundary is locally represented by a $C^1$ function whose derivatives are Dini continuous.
We also prove that strong solutions to oblique derivative problem for elliptic equations in nondivergence form are twice continuously differentiable up to the boundary if the mean oscillations of coefficients satisfy the Dini condition and the boundary is locally represented by a $C^1$ function whose derivatives are double Dini continuous.
This in particular extends a result of M. V. Safonov (Comm. Partial Differential Equations 20:1349--1367, 1995).
\end{abstract}
\maketitle

\section{Introduction and main results}
Let $\Omega$ be a bounded domain in $\bR^n$.
We consider second-order elliptic operators $L$ in divergence form
\begin{equation}							\label{master-d}
L u= \sum_{i,j=1}^n D_i(a^{ij}(x)D_ju+a^i(x)u)+ \sum_{i=1}^n b^i(x) D_i u + c(x)u
\end{equation}
and also second-order elliptic operators $\mathscr{L}$ in nondivergence form
\begin{equation}							\label{master-nd}
\mathscr{L} u= \sum_{i,j=1}^n a^{ij}(x) D_{ij} u+ \sum_{i=1}^n b^i(x) D_i u + c(x)u.
\end{equation}
We assume that the principal coefficients $\mathbf{A}=(a^{ij})_{i,j=1}^n$ are defined on $\bR^n$ and satisfy the uniform ellipticity condition
\begin{equation}					\label{ellipticity}
\lambda \abs{\xi}^2 \le \sum_{i,j=1}^n a^{ij}(x) \xi^i \xi^j, \quad \forall \xi=(\xi^1,\ldots, \xi^n) \in \bR^n,\quad \forall x \in \bR^n
\end{equation}
and the uniform boundedness condition
\begin{equation} \label{boundedness}
\sum_{i,j=1}^n \,\abs{a^{ij}(x)}^2 \le \Lambda^2, \;\; \forall x \in \bR^n
\end{equation}
for some positive constants $\lambda$ and $\Lambda$.
In the nondivergence case, we may assume that $\mathbf{A}$ is symmetric (i.e. $a^{ij}=a^{ji}$) as usual.
We shall further assume that $\mathbf A$ is of Dini mean oscillation; i.e., its mean oscillation function
\[
\omega_{\mathbf A}(r):=\sup_{x \in \bR^n} \fint_{B(x,r)} \,\abs{{\mathbf A}(y)-\bar {\mathbf A}_{x,r}}\,dy \quad \left(\; \bar {\mathbf A}_{x,r} :=\fint_{B(x,r)} \mathbf A\;\right)
\]
satisfies the Dini condition.
We say that a function $\omega: [0,1] \to [0, \infty)$ satisfies the Dini condition if
\[
\int_0^1 \frac{\omega(t)}{t}\,dt <+\infty
\]
and that $\omega$ satisfies the double Dini condition if
\[
\int_0^1 \frac{1}{s} \int_0^s \frac{\omega(t)}{t}\,dt \,ds =\int_0^1 \frac{\omega(t)\ln \frac{1}{t}}{t}\,dt <+\infty.
\]
We say that a function $f$ is Dini continuous (resp. double Dini continuous) if its modulus of continuity satisfies the Dini condition (resp. double Dini condition).
We write $f \in C^{k, \rm{Dini}}$ (resp. $f \in C^{k, \rm{Dini}^2}$) if $D^\alpha f$  is Dini continuous (resp. double Dini continuous) for each multi-index $\alpha$ with $\abs{\alpha} \le k$; refer to Section~\ref{sec_def} for the more precise definitions.

In the divergence case, we assume that $\partial \Omega$ is $C^{1,\rm{Dini}}$ and consider the conormal derivative operator
\[
\mathbf{A} \nabla u \cdot \vec \nu + \vec a u \cdot \vec \nu +a^0 u:= \sum_{i,j=1}^n a^{ij}(x) D_j u \nu^i + \sum_{i=1}^n a^i u \nu^i+a^0 u\quad \text{on}\quad \partial \Omega,
\]
where $\vec \nu=(\nu^1, \ldots, \nu^n)$ denotes the outward unit normal vector, $\vec a= (a^1,\ldots, a^n)$ is of Dini mean oscillation, and $a^0$ is Dini continuous.
In the nondivergence case, we assume that $\partial\Omega$ is $C^{1, \rm{Dini}^2}$ and consider the oblique derivative operator
\[
\beta^0 u + \vec \beta \cdot \nabla u := \beta^0(x) u + \sum_{i=1}^n \beta^i(x) D_i u\quad\text{on}\quad \partial\Omega,
\]
where $\beta^0$ and  $\vec \beta=(\beta^1,\ldots, \beta^n)$ are in $C^{1,\rm{Dini}^2}(\overline \Omega)$ and $\vec \beta$ satisfies the obliqueness condition
\begin{equation}				\label{oblique}
\abs{\vec \beta \cdot \vec \nu}  \ge \mu_0 \, \abs{\vec \beta} \quad\mbox{on}\quad \partial\Omega
\end{equation}
for some positive constant $\mu_0$.

In this paper, we are concerned with the conormal derivative problem for divergence form equation
\[
L u= \dv \vec g + f\;\mbox{ in } \; \Omega,\quad \mathbf{A} \nabla u \cdot \vec \nu+ \vec a u \cdot \vec \nu + a^0 u  = \vec g \cdot \vec \nu+g^0 \;\mbox{ on }\;\partial\Omega,
\]
and the oblique derivative problem for nondivergence form equation
\[
\mathscr{L} u= f\;\mbox{ in } \; \Omega,\quad \beta^0 u + \vec \beta \cdot \nabla u=g\;\mbox{ on }\;\partial\Omega.
\]
For the conormal derivative problem, we shall show that  $u$ is continuously differentiable up to the boundary if the data $\vec g$ is of Dini mean oscillation, $g^0$ is Dini continuous, and if the data $f$ and the lower order coefficients of $L$  belong to $L^q$ with $q>n$.
For the oblique derivative problem, we shall show that $u$ is twice continuously differentiable up to the boundary if the data $f$ and the lower order coefficients of $\mathscr{L}$ are of Dini mean oscillation, and the boundary data $g$, $\beta^0$, and $\vec \beta$ are of $C^{1,\rm{Dini}^2}$.

A few remarks are in order.
Very recently, under the same condition on $\mathbf A$ as imposed here, the first and third named authors \cite{DK17} proved that any $W^{1,2}$ weak solution of the equation $\dv(\mathbf{A} \nabla u)=0$ is continuously differentiable and that any $W^{2,2}$ strong solution of the equation $\tr(\mathbf{A} D^2 u)=0$ is twice continuously differentiable.
Later, the first and third named authors and Escauriaza \cite{DEK17} considered general elliptic equation with lower order coefficients (as considered here) subject to Dirichlet boundary condition and extended the interior estimates in \cite{DK17} to the corresponding $C^1$ and $C^2$ estimates up to the boundary.
In this perspective, this paper can be considered as a natural extension of \cite{DEK17} to conormal and oblique derivative boundary conditions.
Regarding the oblique derivative problem, we are obliged to mention a paper by Safonov \cite{Safonov}, where he proved a priori global $C^{2, \alpha}$ estimates for
solutions assuming that the coefficients and domain satisfy the H\"older condition, which was also established earlier by Lieberman \cite{Lieb87b} by a different method.
We borrowed some crucial technical details from \cite{Safonov} and adapted to our setting.

There are many other literature dealing with the oblique derivative problem and the conormal derivative problem.
Among them, we point out that in \cite[Theorem 5.1]{Lieb87} a result similar to Theorem \ref{thm-main-conormal} below was proved for quasilinear elliptic equations under the uniform Dini continuity condition. In \cite[Theorem 5.4]{Lieb02} a weighted $C^2$ estimate was obtained for fully nonlinear elliptic equations with the oblique derivative boundary condition under the uniform Dini condition.
We also mention a book by Lieberman \cite{Lieb2013}, which gives a comprehensive exposition on the theory of oblique derivative problems for elliptic equations.
We ask readers interested in history and applications of oblique derivative problems to consult \cite{Lieb2013} and references therein.

Now we state the main results of the paper more precisely.
We first consider the conormal derivative problem for a divergence structure elliptic equation.
\begin{condition}			\label{cond-d}
$\mathbf{A}=(a^{ij})$ and $\vec a = (a^1, \ldots, a^n)$ are of Dini mean oscillation in $\overline \Omega$, $a^0$ is Dini continuous in $\overline \Omega$, and $\vec b = (b^1,\ldots, b^n)$ and $c$ belong in $L^q(\Omega)$ with $q>n$.
\end{condition}

\begin{theorem}					\label{thm-main-conormal}
Let $\Omega$ have $C^{1, \rm{Dini}}$ boundary, the coefficients of $L$ in \eqref{master-d} satisfy the conditions \eqref{ellipticity} and \eqref{boundedness}, and Condition~\ref{cond-d}.
Suppose $u\in W^{1,2}(\Omega)$ is a weak solution of
\[
L u= \dv \vec g + f\;\mbox{ in } \; \Omega,\quad \mathbf{A} \nabla u \cdot \vec \nu+ \vec a u \cdot \vec \nu +a^0 u = \vec g \cdot \vec \nu +g^0 \;\mbox{ on }\;\partial\Omega,
\]
where $\vec g = (g^1, \ldots, g^n)$ are of Dini mean oscillation in $\overline\Omega$, $g^0$ is Dini continuous in $\partial\Omega$, and $f \in L^q(\Omega)$ with $q>n.$ Then we have $u\in C^1(\overline \Omega).$
\end{theorem}

We also consider the oblique derivative problem for nondivergence form elliptic equations.

\begin{condition}			\label{cond-nd}
$\mathbf{A}=(a^{ij})$, $\vec b=(b^1,\ldots, b^n)$, and $c$ are of Dini mean oscillation in $\overline \Omega$.
\end{condition}

\begin{condition}			\label{cond-oblique}
$\beta^0$ and $\vec \beta=(\beta^1,\ldots, \beta^n)$ are in $C^{1,\rm{Dini}^2}(\overline \Omega)$, and $\vec \beta$ satisfies \eqref{oblique}.
\end{condition}

\begin{theorem}					\label{thm-main-oblique}
Let $\Omega$ have $C^{1,\rm{Dini}^2}$ boundary, the coefficients of $\mathscr{L}$ in \eqref{master-nd} satisfy the condition \eqref{ellipticity} and \eqref{boundedness}, and Condition~\ref{cond-nd}.
Let $\beta^0$ and $\vec \beta$ satisfy Condition~\ref{cond-oblique}.
Suppose $u\in W^{2,2}(\Omega)$ is a strong solution of the oblique derivative problem
\[
\mathscr{L} u= f\;\mbox{ in } \; \Omega,\quad \beta^0 u + \vec \beta \cdot \nabla u=g\;\mbox{ on }\;\partial\Omega,
\]
where $f$ is of Dini mean oscillation in $\Omega$ and $g \in C^{1, \rm{Dini}^2}(\overline \Omega)$.
Then we have $u\in C^2(\overline \Omega)$.
\end{theorem}

\begin{remark}
In \cite{CM11} global Lipschitz estimates for certain quasilinear divergence form elliptic equations were established under minimal conditions on the data, the nonlinearity, and the domains.
In particular, the condition on the domain is weaker than the $C^{1, \rm{Dini}}$ condition in Theorem \ref{thm-main-conormal}. 
\end{remark}

The organization of the paper is as follows.
In Section~\ref{sec1}, we introduce some notation, definitions, and lemmas used in the paper.
Sections~\ref{sec2} and \ref{sec3} are devoted to the proofs of our main results, Theorem~\ref{thm-main-conormal} and Theorem~\ref{thm-main-oblique}, respectively.
In the Appendix, we provide the proofs for some technical lemmas that are slightly modified from those in Safonov's paper \cite{Safonov}.

\section{Preliminaries}					\label{sec1}
\subsection{Notation and definitions}			\label{sec_def}
We follow the same notation as used in \cite{DEK17}.
For completeness, we reproduce most frequently used ones here.
We denote by $B(x,r)$ the Euclidean ball centered at $x$ with radius $r$ and
\[
B_r = B(0,r), \quad B_r^{+}=B_r \cap \set{x^n >0}\quad\text{and}\quad  T(0,r)=B_r \cap \set{x^n=0}.
\]
Let us fix a smooth domain $\cD$ satisfying
\begin{equation}				\label{setD}
B_{1/2}^{+} \subset \cD \subset B_1^{+}
\end{equation}
so that $\partial \cD$ contains a flat portion $T(0, \frac12)$.
For $\bar x \in \partial \bR^n_{+}=\set{x^n=0}$, we then set
\[
B^{+}(\bar x,r)=B_r^{+}+ \bar x,\quad T(\bar x,r)=T(0,r)+\bar x,\quad\text{and}\quad \cD(\bar x,r)= r \cD+ \bar x.
\]
Hereafter, we shall adopt the usual summation convention for repeated indices.

Throughout the paper, we shall use the notation
\begin{equation}					\label{eq0853m}
[u]_{k; E} := \sup_{x \in E}\, \abs{D^k u(x)} \quad \text{and} \quad [u]_{k,\mu; E} := \sup_{\substack{x, y\in E\\x\neq y}} \frac{\abs{D^k u(x)-D^k u(y)}}{\abs{x-y}^\mu},
\end{equation}
where $k=0, 1, 2, \ldots$, $0<\mu <1$, and $E \subset \bR^n$.
We also write
\begin{equation}					\label{eq0854m}
\abs{u}_{k; E} := \sum_{j=0}^k \, [u]_{j; E} \quad \text{and} \quad \abs{u}_{k,\mu; E} := \abs{u}_{k; E}+ [u]_{k, \mu; E}.
\end{equation}

\begin{definition}				\label{def:dini}
Let $E \subset \bR^n$ and let $f: E \to \bR$.
The modulus of continuity of $f$ is the increasing function $\varrho_f : [0,\infty) \to [0,\infty)$ defined by
\[
\varrho_f(t):=\sup \Set{\,\abs{f(x)-f(y)}: x, y \in E, \; \abs{x-y} \le t\,}.
\]
A function $f$ is said to be Dini continuous (in $E$) if $\varrho_f$ satisfies the Dini condition
\[
\int_0^1 \frac{\varrho_f(t)}{t}\,dt <+\infty;
\]
$f$ is said to be double Dini continuous (in $E$) if $\varrho_f$ satisfies the double Dini condition (see \cite{ME65, ME71})
\[
\int_0^1 \frac{1}{s} \int_0^s \frac{\varrho_f(t)}{t}\,dt\,ds = \int_0^1 \frac{\varrho_f (t)\ln \frac{1}{t}}{t}\,dt <+\infty.
\]
For $k=0,1,2,..$, we denote by $C^{k,\rm{Dini}}(E)$ (resp. $C^{k, \rm{Dini}^2}(E)$) the set of all $k$-times continuously differentiable functions $f$ on $E$ such that $D^\alpha f$ is Dini continuous (resp. double Dini continuous) in $E$, for each multi-index $\alpha$ with $\abs{\alpha} \le k$.
By the $C^{k, \rm{Dini}}$ characteristics of $f$ in $E$, we mean $\abs{f}_{k;E}$ and $\varrho_{D^\alpha f}(t)$ with multi-index $\alpha$ with $\abs{\alpha}=k$.
\end{definition}

\begin{definition}			\label{c1dini}
Let $\Omega(x,r):=\Omega \cap B(x,r)$.
For any $k=1,2,\ldots$, we say that the boundary $\partial \Omega$ is $C^{k,\rm{Dini}}$ (resp. $C^{k,\rm{Dini}^2}$) if for each point $x_0 \in \partial\Omega$, there exist $r>0$ independent of $x_0$ and a $C^{k, \rm{Dini}}$ (resp. $C^{k,\rm{Dini}^2}$) function $\gamma: \bR^{n-1} \to \bR$ such that (upon relabeling and reorienting the coordinates axes if necessary) in a new coordinate system $(x',x^n)=(x^1,\ldots,x^{n-1},x^n)$, $x_0$ becomes the origin and
\[
\Omega(0,r)=\{ x \in B(0, r) : x^n > \gamma(x^1, \ldots, x^{n-1}) \},\quad  \gamma(0')=0,\quad D\gamma(0')=0.
\]
\end{definition}

\begin{remark}           \label{rem2.6}
By using the implicit function theorem and a partition of the unity, it is easily seen that $\partial\Omega$ is of $C^{k,\rm{Dini}}$ (resp. $C^{k,\rm{Dini}^2}$) if and only if there exists a $C^{k, \rm{Dini}}$ (resp. $C^{k,\rm{Dini}^2}$) function $\psi_0: \bR^{n} \to \bR$ such that $\Omega=\{x\in \bR^n\,:\,\psi_0(x)>0\}$ and $|D\psi_0|\ge 1$ on $\partial\Omega$. We call $\psi_0$ a defining function of $\Omega$. Clearly, the $C^{k,\rm{Dini}}$ (resp. $C^{k,\rm{Dini}^2}$) characteristic of $\psi_0$ is comparable to that of $\gamma$ in Definition \ref{c1dini}. In the sequel, we shall use these two equivalent definitions interchangeably.
\end{remark}

\begin{definition}				\label{def:DiniMO}
We say that a function $f: \overline \Omega \to \bR$ is of Dini mean oscillation if its mean oscillation function $\omega_f$ defined by
\[
\omega_f(r):=\sup_{x\in \overline \Omega} \fint_{\Omega(x,r)} \,\abs{f(y)-\bar {f}_{\Omega(x,r)}}\,dy \quad \left(\; \bar f_{\Omega(x,r)} :=\fint_{\Omega(x,r)} f\;\right)
\]
satisfies the Dini condition
\[
\int_0^1 \frac{\omega_f (r)}r \,dr <+\infty.
\]
\end{definition}

\begin{remark}
By H\"older's inequality,  our Dini mean oscillation condition is weaker than the $L^p$-Dini mean oscillation condition for any $p>1$, i.e., the function
\[
\omega_{(p)}(r):=\sup_{x\in \overline \Omega} \left(\fint_{\Omega(x,r)} \,\abs{f(y)-\bar {f}_{\Omega(x,r)}}^p\,dy\right)^{1/p}
\]
is a Dini function.
These conditions are in fact strictly weaker than the uniform Dini continuity condition; see an example in \cite[p. 418]{DK17}.
On the other hand, if we instead consider the functions
\[
\widehat\omega_{(p)}(r):= \sup_{0<s\le r} \omega_{(p)}(s),
\]
then it follows from the proof of \cite[Proposition~1.13]{Acq92} that if $\widehat\omega_{(1)}$ is a Dini function, then $\widehat \omega_{(p)}$ are also Dini functions for all $p \in (1, \infty)$.
It is clear that if $\widehat \omega_{(1)}$ is a Dini function, then $\omega_{(1)}$ is also a Dini function.
However, it is not clear to us whether our Dini mean oscillation condition implies that $\omega_{(p)}$ is a Dini function for $p \in (1,\infty)$.
\end{remark}

Finally, we adopt the usual summation convention over repeated indices.
Also, for nonnegative (variable) quantities $A$ and $B$, the relation $A \lesssim B$ should be understood that there is some constant $c>0$ such that $A \le c B$.
We write $A \simeq B$ if $A \lesssim B$ and $B \lesssim A$.

\subsection{Some preliminary lemmas}
\begin{lemma}						\label{lem1721th}
If $f$ is uniformly Dini continuous and $g$ is of Dini mean oscillation in $\Omega$, then $fg$ is of Dini mean oscillation in $\Omega$.
\end{lemma}
\begin{proof}
For any $x\in \overline\Omega$ and $r>0$, we have
\begin{align*}
\fint_{\Omega(x,r)} \,\Abs{fg-\overline{fg}_{\Omega(x,r)}}
&\le \fint_{\Omega(x,r)} \,\Abs{fg-f\,\bar{g}_{\Omega(x,r)}}
+\fint_{\Omega(x,r)} \,\Abs{f\,\bar{g}_{\Omega(x,r)}-\overline{fg}_{\Omega(x,r)}}\\
&\le \sup_{\Omega(x,r)} f \cdot \omega_g(r) + \varrho_f(r) \cdot \fint_{\Omega(x,r)} \abs{g},
\end{align*}
where we used
\[
\sup_{\Omega(x,r)}\, \Abs{f\,\bar{g}_{\Omega(x,r)}-\overline{fg}_{\Omega(x,r)}}
\le \varrho_f(r) \cdot \fint_{\Omega(x,r)} \abs{g}.
\]
Therefore, we get
\[
\omega_{fg}(r) \le \norm{f}_{L^\infty(\Omega)}\, \omega_g(r)+\norm{g}_{L^\infty(\Omega)}\,\varrho_f(r)
\]
and thus $\omega_{fg}$ is a Dini function.
\end{proof}

\begin{lemma}						\label{lem1822sat}
Let $\omega: [0, a] \to [0, \infty)$ be a function satisfying the (double) Dini condition.
Suppose there are constants $c_1, c_2 >0$ such that
\begin{equation}				\label{eq1925sat}
c_1 \omega(t) \le \omega(s) \le c_2 \omega(t)
\end{equation}
whenever $\tfrac12 t \le s \le  t$ and $0\le t \le a$. (It should be noted that the condition \eqref{eq1925sat} is automatically satisfied by $\varrho_f(t)$ and $\omega_f(t)$ introduced in Definitions~\ref{def:dini} and \ref{def:DiniMO}).
Let $\beta \in (0, 1]$ be given.
Then, there is a function $\tilde \omega: [0,a] \to [0,\infty)$ such that $\omega(t) \le \tilde \omega(t)$ for any $t \in [0,a]$ and that $t \mapsto t^{-\beta} \tilde \omega(t)$ is decreasing on $(0, a]$.
Moreover, $\tilde \omega(t)$ satisfies the (double) Dini condition and also satisfies the condition \eqref{eq1925sat}.
\end{lemma}
\begin{proof}
We set  $\tilde \omega(0)=0$ and for $0<t \le a$, define
\[
\tilde \omega(t) =\sup_{s \in [t, a]} \left(\frac{t}{s}\right)^\beta \omega(s).
\]
Then, it is clear that $\omega(t) \le \tilde\omega(t)$ and $t\mapsto \tilde \omega(t)/t^\beta$ is decreasing.
Also, it is straightforward to verify that $\tilde \omega$ satisfies \eqref{eq1925sat} when $\omega$ satisfies \eqref{eq1925sat}.
Finally, we refer to the proof of \cite[Lemma~2.9]{DEK17} for the fact that $\tilde\omega(t)$ satisfies the (double) Dini condition.
\end{proof}

\begin{lemma}						\label{lem2.5}
Let $\cD \subset \bR^n$ be a smooth domain satisfying \eqref{setD}.
Let $\bar{\mathbf A} = (\bar a^{ij})$ be a constant matrix satisfying \eqref{ellipticity} and \eqref{boundedness}.
For $\vec f \in L^2(\cD),$ let $u \in W^{1,2}(\cD)$ be a weak solution of
\[
\dv (\bar {\mathbf A} \nabla u)= \dv \vec f \;\mbox{ in } \; \cD, \quad \bar{\mathbf A} \nabla u \cdot \vec \nu = \vec f \cdot \vec \nu\;\mbox{ on }\; \partial \cD.
\]
Then there exists a constant $C=C(n,\lambda, \Lambda, \cD)$ such that for any $t>0$, we have
\[
\Abs{\set{x \in \cD : \abs{Du(x)}>t}} \le \frac{C}{t}\int_{\cD} \abs{\vec f}.
\]
\end{lemma}
\begin{proof}
Since $u$ is unique up to a constant, we see that the map $T: \vec f \mapsto Du$ is well defined and is a bounded linear operator on $L^2(\cD)$.
We modify the proof of \cite[Lemma~2.2]{DK17} using \cite[Lemma~4.1]{DEK17}.
Let $\vec b \in L^2(\cD)$ be supported in $B(\bar y, r) \cap \cD$ with mean zero, where $\bar y \in \cD$ and $0<r< \frac12 \diam \cD$.
Suppose $u \in W^{1,2}(\cD)$ is a weak solution (unique up to a constant) of
\[
\dv (\bar {\mathbf A} \nabla u)= \dv \vec b\;\text{ in } \; \cD, \quad \bar{\mathbf A} \nabla u \cdot \vec \nu = \vec b \cdot \vec \nu\;\mbox{ on }\; \partial \cD.
\]
By \cite[Lemma~4.1]{DEK17}, it is enough to show that
\[
\int_{\cD\setminus B(\bar y, 2r)} \abs{Du} \le C \int_{B(\bar y, r)\cap \cD} \abs{\vec b}.
\]
For any $R \ge 2r$ such that $\cD \setminus B(\bar y, R) \neq \emptyset$ and $\vec g \in C_c^{\infty}((B(\bar y, 2R) \setminus B(\bar y, R)) \cap \cD),$ let $v \in W^{1,2}(\cD)$ be a weak solution (unique up to a constant) of
\[
\dv (\bar {\mathbf A}\tran \nabla v)= \dv \vec g\;\text{ in } \; \cD, \quad \bar{\mathbf A}\tran \nabla v \cdot \vec \nu = \vec g \cdot \vec \nu\;\mbox{ on }\; \partial \cD.
\]
Then, we have the following equality
\[
\int_{\cD}Du \cdot \vec g = \int_{\cD} \vec b \cdot Dv = \int_{B(\bar y, r) \cap \cD}\vec b \cdot (Dv- \overline{Dv}_{B(\bar y, r) \cap \cD}).
\]
Therefore we get, by the mean value theorem,
\begin{align*}
\Abs{\int_{(B(\bar y,2R)\setminus B(\bar y, R)) \cap \cD}Du \cdot \vec g} &\le \norm{\vec b}_{L^1(B(\bar y, r) \cap \cD)}\norm{Dv- \overline{Dv}_{B(\bar y, r) \cap \cD}}_{L^{\infty}(B(\bar y, r)\cap \cD)} \\
& \le 2r \norm{\vec b}_{L^1(B(\bar y, r) \cap \cD)}\norm{D^2v}_{L^{\infty}(B(\bar y, r)\cap \cD)}.
\end{align*}
Note that $\dv (\bar {\mathbf A}\tran \nabla v) =0$ in $B(\bar y, R) \cap \cD$ and $r \le \frac12 R$.
Since $\bar {\mathbf A}$ is constant and the boundary $\partial \cD$ is smooth, we have
\begin{align*}
\norm{D^2v}_{L^{\infty}(B(\bar y, r) \cap \cD)} &\le CR^{-1-\frac{n}{2}}\norm{Dv}_{L^2(B(\bar y, R) \cap \cD)} \le CR^{-1-\frac{n}{2}}\norm{Dv}_{L^2(\cD)} \\
& \le CR^{-1-\frac{n}{2}}\norm{\vec g}_{L^2(\cD)} = CR^{-1-\frac{n}{2}}\norm{\vec g}_{L^{2}((B(\bar y,2R)\setminus B(\bar y, R)) \cap \cD)}.
\end{align*}
Therefore, by the duality, we have
\[
\norm{Du}_{L^2((B(\bar y, 2R)\setminus B(\bar y, R)) \cap \cD)} \le CrR^{-1-\frac{n}{2}}\norm{\vec b}_{L^1(B(\bar y, r)\cap \cD)}
\]
and hence by H\"older's inequality we get
\[
\norm{Du}_{L^1((B(\bar y, 2R)\setminus B(\bar y, R)) \cap \cD)} \le CrR^{-1}\norm{\vec b}_{L^1(B(\bar y, r)\cap \cD)}.
\]
Now let $N>0$ be the smallest positive integer such that $\cD \subset B(\bar y, 2^{N+1}r)$.
By taking $R = 2r, 4r,\ldots, 2^{N}r$ in the above, we get
\[
\int_{\cD \setminus B(\bar y, 2r)} \abs{Du} \le  C\sum_{k=1}^N2^{-k}\norm{\vec b}_{L^1(B(\bar y, r) \cap \cD)} \le C \int_{B(\bar y, r) \cap \cD}\abs{\vec b}.
\]
We note that $C$ depends only on $n$, $\lambda$, $\Lambda$, and $\cD$.
Thus, we see that $T$ satisfies the hypotheses of Lemma~4.1 of \cite{DEK17}, and the proof is complete.
\end{proof}

\begin{lemma}					\label{lem2.6}
Let $\bar{\mathbf A} = (\bar a^{ij})$ be a constant symmetric matrix satisfying \eqref{ellipticity} and \eqref{boundedness}.
For $f \in L^2(B^{+}_1)$, let $u \in W^{2,2}(B^{+}_1)$ be a strong solution of the mixed problem
\begin{equation}				\label{mixed_prob}
\bar a^{ij} D_{ij} u=f\;\text{ in }\; B^{+}_1,\quad u=0\;\text{ on }\; \partial B_1 \cap \set{ x^n >0},\quad D_n u=0  \;\text{ on }\; T(0,1).
\end{equation}
Then there exists a constant $C=C(n,\lambda, \Lambda)$ such that for any $t>0$, we have
\[
\Abs{\set{x \in B^{+}_1 : \abs{D^2 u(x)} >t}}  \le \frac{C}{t} \int_{B^{+}_1} \abs{f}.
\]
\end{lemma}
\begin{proof}
Without loss of generality, we may assume that $\bar a^{nn}=1$.
We introduce a new matrix valued function $\hat{\mathbf A}=\hat{\mathbf A}(x^n)$ as follows.
When $i=j=n$ or $i, j \in \set{1,\ldots, n-1}$,
\[
\hat a^{ij}(x^n) = \bar a^{ij}.
\]
When $j=1, \ldots, n-1$,
\[
\hat a^{nj}(x^n)=\hat a^{jn}(x^n) =
\begin{cases}
\bar a^{nj} \; &\text{if} \;\; x^n \ge 0, \\
-\bar a^{nj} \; &\text{if} \;\; x^n <0.
\end{cases}
\]
It is easy to check that $\hat{\mathbf A}$ satisfies the conditions \eqref{ellipticity} and \eqref{boundedness}.
Let $\hat f$ be an even extension of $f$ and let $\hat u \in W^{2,2}(B_1) \cap W^{1,2}_0(B_1)$ be a unique solution of
\begin{equation}					\label{eq1550w}
\hat a^{ij} D_{ij}  u= \hat f \;\mbox{ in } \; B_1, \quad u=0 \;\mbox{ on }\; \partial B_1.
\end{equation}
See \cite{DongKrylov10} for the solvability of \eqref{eq1550w}.
By the uniqueness, it is straightforward to see that $\hat u$ is even with respect to $x^n$ coordinate, which implies that $D_n \hat u= 0$ on $T(0,1)$.
Then by the uniqueness of the mixed problem \eqref{mixed_prob}, we conclude that $u \equiv \hat u$ in $B^{+}_1$.
Therefore, it is enough to show
\[
\Abs{\set{x \in B_1 : \abs{D^2 \hat u(x)} >t}}  \le \frac{C}{t} \int_{B_1} \abs{\hat f}.
\]

Fix $\bar y \in B_1$, $0<r<\frac12$, and let $b \in L^2(B_1)$ be supported in $B(\bar y, r) \cap B_1$ with mean zero.
Let $\tilde u \in W^{2,2}(B_1) \cap W^{1,2}_0(B_1)$ be a solution of
\begin{equation}					\label{eq1400sa}
\hat a^{ij} D_{ij} u= b\;\mbox{ in }\; B_1;\quad
u=0 \;\mbox{ on } \; \partial B_1,
\end{equation}
the solvability of which is stated in \cite[p. 6483]{DongKrylov10}.

For any $R\ge 2r$ such that $B_1\setminus B(\bar y, R) \neq \emptyset$ and $\mathbf{g}=(g^{kl}) \in C^\infty_c((B(\bar y,2R)\setminus B(\bar y,R))\cap B_1)$, let $v \in W^{1,2}_0(B_1)$  be a weak solution of
\[
D_i(\tilde a^{ij} D_j v) = \dv^2\mathbf g \;\mbox{ in } \; B_1, \quad v=0 \;\mbox{ on }\; \partial B_1,
\]
where $\tilde{\mathbf A}=(\tilde a^{ij})$ is defined as follows
\begin{gather*}
\tilde a^{nn}=1;\quad
\tilde a^{ij}= \hat a^{ij}\;\text{ for }\; i, j \in \set{1,\ldots, n-1};\\
\tilde a^{nj}=2\hat a^{nj}\;\text{ and }\;\tilde a^{jn}=0\;\text{ for }\;j=1, \ldots, n-1.
\end{gather*}
It is easy to check that $\tilde{\mathbf A}$ satisfies the ellipticity and boundedness conditions \eqref{ellipticity} and \eqref{boundedness} (with new constants $\tilde \lambda$ and $\tilde \Lambda$ determined by $\lambda$ and $\Lambda$).
Since $\mathbf{g}= 0$ in $B(\bar y, R) \cap B_1$ and $r\le R/2$, we find
\[
D_i(\tilde a^{ij} D_j v) = 0 \quad\text{in}\quad B(\bar y, R) \cap B_1,
\]
and thus, by the De Giorgi-Nash-Moser estimate (up to the boundary) we see that $v$ is  H\"older continuous in $B(\bar y, r) \cap B_1$ and
\begin{equation}					\label{eq1421sa}
[v]_{\mu;B(\bar y,r) \cap B_1} \le C R^{-\mu-\frac{n}2}\norm{v}_{L^2(B(\bar y, R) \cap B_1)}
\end{equation}
for some constants $\mu \in (0,1)$ and $C>0$ depending only on $n$, $\lambda$, and $\Lambda$.

On the other hand, observe that
\begin{align*}
\sum_{i,j =1}^n D_{i}(\tilde a^{ij}D_jv) &= \sum_{i,j=1}^{n-1}D_{i} (\hat a^{ij} D_j v) + 2 \sum_{j=1}^{n-1}D_{n}(\hat a^{nj}D_jv) +D_{n}(D_n v)\\
&= \sum_{i,j=1}^{n-1}D_{ij} (\hat a^{ij} v) + 2 \sum_{j=1}^{n-1}D_{nj}(\hat a^{nj}v) +D_{nn}v = \sum_{i,j =1}^n D_{ij}(\hat a^{ij}v).
\end{align*}
Here, we used that $\hat a^{ij}=\hat a^{ij}(x^n)$ and $\hat a^{nn}=1$.
Therefore, we see that $v$ is also an adjoint solution of
\begin{equation}					\label{eq1616w}
D_{ij}(\hat a^{ij} v) = \dv^2\mathbf g \;\mbox{ in } \; B_1, \quad v=0 \;\mbox{ on }\; \partial B_1
\end{equation}
and hence by \cite[Lemma 2]{EM17}, we have
\begin{equation}					\label{eq1423sa}
\norm{v}_{L^2(B_1)} \le C \norm{\mathbf g}_{L^2(B_1)}.
\end{equation}
By \eqref{eq1400sa} and \eqref{eq1616w} and the hypothesis on $b$, we have the identity
\[
\int_{B_1} D_{ij} \tilde u \, g^{ij} = \int_{B_1} vb = \int_{B(\bar y,r) \cap B_1}b(v - \bar v_{B(\bar y, r) \cap B_1}).
\]
Then by using \eqref{eq1421sa} and \eqref{eq1423sa}, we have
\begin{align*}
\int_{(B(\bar y, 2R) \setminus B(\bar y, R)) \cap B_1} D_{ij} \tilde u\, g^{ij} &\le \norm{b}_{L^1(B(\bar y, r)\cap B_1)} [v]_{\mu;B(\bar y,r) \cap B_1} (2r)^{\mu} \\
& \le C\left(\frac{r}R\right)^{\mu}R^{-\frac{n}2}\norm{b}_{L^1(B(\bar y, r)\cap B_1)}\norm{\mathbf g}_{L^2((B(\bar y, 2R) \setminus B(\bar y, R)) \cap B_1)}.
\end{align*}
The rest of the proof is almost the same as that of Lemma~\ref{lem2.5} and omitted.
\end{proof}

\begin{lemma}					\label{lem2.7}
Let $\bar{\mathbf A} = (\bar a^{ij})$ be a constant symmetric matrix satisfying \eqref{ellipticity} and \eqref{boundedness}.
Suppose $u\in W^{2,2}(B_1^{+})$ satisfies
\[
\bar a^{ij} D_{ij} u=0\;\text{ in }\; B^{+}_1,\quad D_n u=0  \;\text{ on }\; T(0,1).
\]
Then for any $p>0$, there exists a constant $C=C(n,\lambda, \Lambda, p)$ such that
\begin{equation}				\label{eq0957th}
\norm{D u}_{L^\infty(B_{1/2}^{+})} \le C  \left( \fint_{B_1^{+}}\, \abs{u}^p\right)^{\frac{1}{p}}.
\end{equation}
\end{lemma}
\begin{proof}
The estimate \eqref{eq0957th} can be deduced from \cite[Theorem~6.26]{GT}.
We give an alternative proof here.
Let $\hat u$ be an even extension of $u$ (with respect to $x^n$ coordinate) and $\hat{\mathbf A}$ be defined as in the proof of Lemma~\ref{lem2.6}.
Then $\hat u$ satisfies
\[
\hat a^{ij} D_{ij} \hat u =0 \quad\text{in} \quad B_1.
\]
Since $\hat{\mathbf A}=\hat{\mathbf A}(x^n)$, we have the Lipschitz estimate (see \cite{KimDKrylov07})
\[
\norm{D \hat u}_{L^\infty(B_{1/2})} \le C \norm{\hat u}_{L^2(B_{1})},
\]
from which \eqref{eq0957th} follows by standard argument.
\end{proof}

\section{Proof of Theorem~\ref{thm-main-conormal}}			\label{sec2}
We begin with the following proposition dealing with interior $C^1$ estimates.

\begin{proposition}			\label{prop3.1}
We have $u \in C^1(\overline{\Omega'})$ for any $\Omega'\subset\subset \Omega$.
\end{proposition}
\begin{proof}
Since the proof is very similar to that of \cite[Proposition~2.6]{DEK17}, we will only give an outline of the proof.
Since the coefficients $a^{ij}$ are continuous, the standard $W^{1,p}$ theory yields that $u\in W^{1,p}_{\loc}(\Omega)$ for any $p\in (1,\infty)$.
To see that $u\in W^{1,p}$ up to the boundary, we locally flatten the boundary so that $\vec \nu=-\vec e_n$ and the boundary condition becomes
\[
-\sum_{j=1}^n a^{nj}D_j u-a^n u =-g^n+g^0-a^0 u\quad \text{on}\quad  \Gamma \subset \set{x^n=0}.
\]
Note that if we set
\[
\tilde g^n(x)= \tilde g^n(x',x^n):=g^n(x',x^n)-g^0(x',0)+a^0(x',0)u(x',0),
\]
then we have $-D_n \tilde g^n=-D_n g^n$.
Therefore, by replacing $g^n$ with $\tilde g^n$, the above boundary condition reduces to the standard conormal boundary condition (see e.g.,  \cite[Theorem 5]{DKD10}).
Then we can apply the boundary $W^{1,p}$ theory and a bootstrap argument to conclude that $u\in W^{1,p}(\Omega)$ for any $p\in (1,\infty)$.

By the Morrey-Sobolev embedding, we have $u \in C^{0,\mu}(\overline \Omega)$ for any $\mu\in (0, 1)$.
Rewriting the equation, we have
\[
D_i(a^{ij}D_ju) = f - b^i D_i u - cu + D_i(g^i- a^i u).
\]
Let $\vec g' = \vec g - \vec a u$.
By Lemma~\ref{lem1721th}, we see that $\vec g'$ is of Dini mean oscillation.
Also, by taking a sufficiently large $p$ and using H\"older's inequality, we have $f-b^i D_i u-cu \in L^r(\Omega)$ for some $r\in (n, q)$.
We set $\vec g'' = \nabla v$, where $v$ solves
\[
\Delta v = f - b^i D_i u - cu \;\mbox{ in } \; \Omega,\quad \partial v/ \partial \nu = 0\;\mbox{ on }\;\partial\Omega.
\]
Then, we have $\vec g'' \in C^{0,\delta}(\overline \Omega)$ with $\delta = 1-\frac{n}{r}$.
Therefore, we see that $\vec g'$ and $\vec g''$ are of Dini mean oscillation and
\[
D_i(a^{ij}D_ju) =  \dv(\vec g' + \vec g'')\;\mbox{ in } \; \Omega.
\]
By \cite[Theorem 1.5]{DK17}, we conclude that $u \in C^1(\overline{\Omega'})$ for any $\Omega' \subset\subset \Omega$.
\end{proof}

Next, we prove $C^1$ estimate near the boundary.
Note that $\vec g''$ introduced in the proof of Proposition~\ref{prop3.1} satisfies $\vec g'' \cdot \vec \nu =0$ on $\partial \Omega$.
By the same reasoning explained just before \cite[Proposition~2.7]{DEK17} and replacing $g^n$ by $\tilde g^n$ after locally flattening the boundary so that $\vec \nu=-\vec e_n$ (note that $\tilde g^n$ is of Dini mean oscillation), we are reduced to prove the following.
\begin{proposition}					\label{prop3.2}
If $u \in W^{1,2}(B_4^+)$ is a weak solution of
\[
D_i(a^{ij} D_ju)= \dv \vec g \;\; \text{ in } \;\; B_4^{+}, \quad \mathbf{A} Du \cdot \vec e_n = \vec g \cdot \vec e_n \;\;\text{ on }\;\; T(0,4),
\]
then $u \in C^1(\overline B{}^+_1)$.
\end{proposition}

The rest of this section is devoted to the proof of Proposition~\ref{prop3.2}.
We shall assume $u \in C^1(\overline B{}^+_3)$ and derive an a priori estimate of the modulus of continuity of $Du$.
We fix some $p \in (0,1)$ and introduce
\[
\phi(x,r) :=\inf_{\vec q \in \bR^n} \left(\fint_{B(x,r)\cap B_4^+} \abs{Du-\vec q}^p \right)^{\frac1p}.
\]

We shall derive an estimate for $\phi(\bar x, r)$ for $\bar x \in T(0,3)$ and $0< r \le \frac12$.
Recall the notation $\cD(\bar x, r)$ introduced at the beginning of this section.
We split $u=v+w$, where $w \in W^{1,2}(\cD(\bar x, 2r))$ is a weak solution of the problem
\begin{align*}
\dv (\bar{\mathbf A} \nabla w) &= - \dv((\mathbf{A} - \bar{\mathbf A})\nabla u) +\dv(\vec g - \bar{\vec g}) \;\text{ in } \; \cD(\bar x, 2r),\\
\bar {\mathbf A} \nabla w \cdot \vec \nu& = -(\mathbf{A} - \bar{\mathbf A})\nabla u \cdot \vec \nu +(\vec g - \bar{\vec g})\cdot \vec \nu \;\text{ on }\; \partial \cD(\bar x, 2r),
\end{align*}
where $\bar{\mathbf A} = \bar{\mathbf A}_{B^{+}(\bar x, 2r)}$ and $\bar{\vec g} = \bar{\vec g}_{B^{+}(\bar x, 2r)}$.
By Lemma~\ref{lem2.5} with scaling, we see that
\[
\Abs{\set{x \in B^{+}(\bar x,r) :  \abs{Dw(x)}> t}} \le \frac{C}{t} \left(\, \norm{Du}_{L^{\infty}(B^{+}(\bar x, 2r))}\int_{B^{+}(\bar x, 2r)}\abs{\mathbf A-\bar{\mathbf A}}+\int_{B^{+}(\bar x, 2r)}\abs{\vec g-\bar {\vec g}}\, \right).
\]
Then, we have (see \cite[(2.11)]{DK17})
\begin{equation}				\label{eq1323th}
\left(\fint_{B^{+}(\bar x,r)} \abs{Dw}^p\right)^{\frac1p} \le C\omega_{\mathbf A}(2r)\,\norm{Du}_{L^{\infty}(B^{+}(\bar x, 2r))}+C\omega_{\vec g}(2r).
\end{equation}
Note that $v= u-w$ satisfies
\[
\dv( \bar{\mathbf  A} \nabla v)= \dv \vec {\bar g}=0 \;\mbox{ in } \; B^{+}(\bar x, r), \quad \bar {\mathbf A} \nabla v \cdot \vec e_n = \bar{\vec g}\cdot \vec e_n \;\mbox{ on }\; T(\bar x, r).
\]
Then for any $c \in \bR$ and $k=1, 2, \ldots, n-1$, $\tilde v= D_k v - c$
satisfies
\[
\dv( \bar{\mathbf  A} \nabla \tilde v)= 0 \;\mbox{ in } \;  B^{+}(\bar x, r), \quad \bar {\mathbf A} \nabla \tilde v \cdot \vec e_n = 0 \;\mbox{ on }\; T(\bar x, r).
\]
By the standard elliptic estimates for the constant coefficients equations with zero conormal boundary data, we have
\[
\norm{DD_kv}_{L^\infty(B^+(\bar x, \frac12 r))} \le Cr^{-1}\left( \fint_{B^+(\bar x, r)} \abs{D_{k}v-c}^p \right)^{\frac1p},\quad  \forall k \in \set{1, \ldots, n-1},\;\;\forall c \in \bR.
\]
Then by using $\displaystyle D_{nn}v=-\frac{1}{\bar{a}^{nn}}\sum_{(i,j)\neq (n,n)} \bar{a}^{ij}D_{ij}v$,
we obtain
\[
\norm{D^2v}_{L^{\infty}(B^+(\bar x, \frac12 r))} \le C \norm{D D_{x'} v}_{L^{\infty}(B^+(\bar x, \frac12 r))} \le Cr^{-1}\left( \fint_{B^+(\bar x, r)} \abs{D_{x'}v-\vec c}^p \right)^{\frac1p},\quad \forall \vec c \in \bR^{n-1},
\]
where we used the notation $D_{x'} v= (D_1 v, \ldots, D_{n-1} v)$.
Therefore, we have
\[
\norm{D^2v}_{L^{\infty}(B^+(\bar x, \frac12 r))} \le Cr^{-1}\left( \fint_{B^+(\bar x, r)} \abs{Dv-\vec q}^p \right)^{\frac1p},\quad \forall \vec q \in \bR^{n}.
\]
Let $0 < \kappa <\frac12$ be a number to be fixed later.
Since
\[
\left( \fint_{B^+(\bar x, \kappa r)} \Abs{Dv-\overline{Dv}_{B^+(\bar x, \kappa r)}}^p \right)^{\frac1p} \le 2\kappa r \norm{D^2 v}_{L^{\infty}(B^+(\bar x, \kappa r))}
\]
and $\kappa < \frac12$, we see that there is a constant $C_0=C_0(n, \lambda, \Lambda, p)>0$ such that
\[
\left( \fint_{B^+(\bar x, \kappa r)} \Abs{Dv-\overline{Dv}_{B^+(\bar x, \kappa r)}}^p \right)^{\frac1p} \le C_0\kappa \left( \fint_{B^+(\bar x, r)} \abs{Dv-\vec q}^p \right)^{\frac1p},\quad \forall \vec q \in \bR^n.
\]
By using the decomposition $u = v+w,$ we obtain from the above that
\begin{align*}
&\left( \fint_{B^+(\bar x, \kappa r)}\Abs{Du-\overline{Dv}_{B^+(\bar x, \kappa r)}}^p \right)^{\frac1p}\\
&\qquad \qquad  \le 2^{\frac{1-p}{p}}\left( \fint_{B^+(\bar x, \kappa r)} \Abs{Dv-\overline{Dv}_{B^+(\bar x, \kappa r)}}^p \right)^{\frac1p}+2^{\frac{1-p}{p}} \left( \fint_{B^+(\bar x, \kappa r)} \abs{Dw}^p \right)^{\frac1p}\\
& \qquad \qquad \le 4^{\frac{1-p}{p}}C_0 \kappa \left( \fint_{B^{+}(\bar x, r)} \abs{Du-\vec q}^p \right)^{\frac1p}+C(\kappa^{-\frac{n}p}+1)\left( \fint_{B^{+}(\bar x, r)} \abs{Dw}^p \right)^{\frac1p}. 	\end{align*}
Since $\vec q \in \bR^n$ is arbitrary, by using \eqref{eq1323th}, we obtain
\begin{equation}				\label{eq0750f}
\phi(\bar{x},\kappa r) \le 4^{\frac{1-p}{p}}C_0 \kappa\,\phi(\bar x, r) + C(\kappa^{-\frac{n}p}+1)\left( \omega_{\mathbf{A}}(2r)\norm{Du}_{L^{\infty}(B^+(\bar x, 2 r))}+\omega_{\vec g}(2r) \right).
\end{equation}
Therefore, we see that $\phi(\bar x, r)$ enjoys the same estimates for the auxiliary quantity $\varphi(\bar x, r)$ defined in \cite[(2.10)]{DEK17} for Dirichlet boundary problem, and thus \cite[Lemma 2.8]{DEK17} is valid with $\phi(\bar x, \rho)$ in place of $\varphi(\bar x, \rho)$.
Also, we note that if $B(x, \rho) \subset B^{+}(\bar x, R)$ and $\rho \simeq R$, then we have (see (2.26) and (2.28) in \cite{DEK17})
\begin{equation}					\label{eq1209f}
\phi(x, \rho) \lesssim \phi(\bar x, R).
\end{equation}
Consequently, we have Lemmas 2.9, 2.11, and 2.12 in \cite{DEK17} also available in our setting.
For any given $\beta \in (0,1)$, by taking $\kappa \in (0,\frac12)$ sufficiently small such that $4^{\frac{1-p}{p}}C_0 \kappa\le \kappa^\beta$, we thus have the following estimate
\begin{multline}						\label{eq1707m}
\abs{Du(x)-Du(y)} \le  C\norm{Du}_{L^1(B_2^+)}\,\abs{x-y}^\beta \\
+C\left( \norm{Du}_{L^1(B_4^+)} + \int_0^1 \frac{\hat\omega_{\vec g}(t)}t\,dt \right)  \omega^*_{\mathbf{A}}(\abs{x-y})+ C\omega^*_{\vec g}(\abs{x-y}),\quad \forall x, y \in B^{+}_1.
\end{multline}
Here, $\hat\omega_{\vec g}$ is a function determined by $\omega_{\vec g}$
satisfying the Dini condition and $\omega_{\mathbf A}^*(t)$, and $\omega_{\vec g}^*(t)$ are as defined by the formula \cite[(2.34)]{DEK17}.
More precisely, in \cite{DEK17} we defined
\begin{align*}
\tilde \omega_{\bullet}(t)&:= \sum_{i=1}^{\infty} \kappa^{i \beta} \left(\omega_{\bullet}(\kappa^{-i}t)\, [\kappa^{-i}t \le 1] + \omega_{\bullet}(1)\, [\kappa^{-i}t >1] \right),\\
\omega^{\sharp}_{\bullet}(t) &:= \sup_{s \in [t,1]} \,\left(\frac{t}{s}\right)^\beta\, \tilde\omega_{\bullet}(s)\quad \text{for }\;0<t\le 1,\\
\hat\omega_{\bullet}(t)&:= \tilde \omega_{\bullet}(t)+\tilde\omega_{\bullet}(4t)+\omega_{\bullet}^{\sharp}(4t)\;
(~\lesssim \omega_{\bullet}^{\sharp}(4t)\;) \quad \text{for }\;0<t\le 1/4,\\
\omega^*_\bullet(t)&:= \hat\omega_\bullet(t)+ \int_0^t \frac{\tilde \omega_\bullet(s)}s \,ds +\tilde \omega_\bullet(4t)+ \int_0^t \frac{\tilde \omega_\bullet(4s)}s \,ds\quad \text{for} \quad 0<t\le 1/4.
\end{align*}
In particular, they satisfy
\[
\lim_{t \to 0+} \omega_{\bullet}^*(t) = 0.
\]
This completes the proof of Proposition~\ref{prop3.2} and that of Theorem~\ref{thm-main-conormal}.
\qed

\begin{remark}
In the case when $\mathbf A$ and $\vec g$ are H\"older continuous with an exponent $\alpha \in (0,1)$, then by choosing $\beta \in (\alpha,1)$ in \eqref{eq1707m}, one can show that  $\omega_{\mathbf A}^*(t) \lesssim t^\alpha$ and $\omega_{\vec g}^*(t) \lesssim t^\alpha$.
Therefore, $Du$ is H\"older continuous with the same exponent $\alpha$, which recovers the classical result.
\end{remark}

\section{Proof of Theorem~\ref{thm-main-oblique}}			\label{sec3}

We shall use the term ``the prescribed data'' collectively for the following:
the dimension $n$, ellipticity constants $\lambda$, $\Lambda$, the obliqueness constant $\mu_0$, and the domain $\Omega$; the mean oscillation functions for the coefficients $\omega_{\mathbf A}$, $\omega_{\vec b}$, and $\omega_c$, all of which satisfy the Dini condition;
$C^{1,\rm{Dini}^2}$ characteristics of the coefficients $\beta^0$, $\vec \beta$ in the oblique derivative operator and those of $\gamma$ (or equivalently $\psi_0$ in Remark~\ref{rem2.6}), which (locally) represents the boundary.

The following proposition provides key estimates, the proof of which is deferred to the end of this section.
\begin{proposition}			\label{prop4.1}
If $u \in W^{2,2}(B_4^+)$ is a strong solution of
\[
a^{ij} D_{ij}u= f \;\mbox{ in } \; B_4^{+}, \quad D_n u=0 \;\mbox{ on }\; T(0,4),
\]
then we have
\begin{equation}					\label{eq0830f}
[u]_{2; B^{+}_2} \le C \norm{D^2u}_{L^1(B^{+}_4)} + C \int_0^1 \frac{\hat \omega_f(t)}{t} dt,
\end{equation}
where as in the previous section $\hat \omega_f(t)$ is a nonnegative function derived from $\omega_f(t)$ satisfying the Dini condition.
Moreover, for any $x, y \in B_1^{+}$, we have
\begin{equation}					\label{eq1131m}
\abs{D^2u(x)-D^2u(y)} \le  C \left\{\norm{D^2u}_{L^1(B_2^+)}\,\abs{x-y}^\mu
+[u]_{2; B^{+}_2} \, \omega^*_{\mathbf{A}}(\abs{x-y})+ \omega^*_{f}(\abs{x-y}) \right\},
\end{equation}
where $\mu \in (0,1)$ is any number, $C=C(n, \lambda, \Lambda, \omega_{\mathbf A}, \mu)$, and $\omega_{\bullet}^{*}(t)$ is a modulus of continuity determined by  $\omega_{\bullet}(t)$ and $\mu$, which goes to zero as $t\to 0$.
\end{proposition}

We take the proposition for now.
Our first goal is to establish the following estimate under a qualitative assumption that $u \in C^2(\overline \Omega)$:
\begin{equation}					\label{eq0858m}
[u]_{2; \Omega} \le C \left( \norm{u}_{W^{2,1}(\Omega)} +  \int_0^{r_0} \frac{\tilde \omega_f(t)}{t}dt + \abs{g}_{1;\Omega}+ \int_0^{r_0} \frac{\tilde \varrho_{Dg}(t)\ln \frac{1}{t}}{t}dt\right),
\end{equation}
where $r_0>0$ and $C$ are constants depending on the prescribed data;
$\tilde \omega_f(t)$ satisfies the Dini condition and is determined by $\omega_f(t)$ and the prescribed data;
$\tilde\varrho_{Dg}(t)$ satisfies the double Dini condition and is determined by $\varrho_{Dg}(t)$ and the prescribed data.

With the estimate \eqref{eq0858m} at hand, we then show that for any $x, y \in \Omega$, we have
\begin{align}
							\nonumber
\abs{D^2u(x)-D^2u(y)} &\le  C \norm{D^2u}_{L^1(\Omega)}\,\abs{x-y}^\mu+ C[u]_{2; \Omega} \,\omega^*_{\mathbf{A}}(\abs{x-y})+C\omega^*_{f}(\abs{x-y})\\
							\nonumber		
&\quad+ C \left(\abs{g}_{1;\Omega}+ \int_0^{r_0} \frac{\tilde \varrho_{Dg}(t)\ln \frac{1}{t}}{t}dt\right)  \omega_0^*(\abs{x-y})\\
						\label{eq0906m}
&\quad +C\int_0^{\abs{x-y}} \frac{\tilde \varrho_{D g}(t)\ln \frac{1}{t}}{t}dt +
C \omega_1^{*}(\abs{x-y}).
\end{align}
Here, $\mu \in (0,1)$ is an arbitrary constant, $C$ is a constant depending on $\mu$ and the prescribed data;\, $\omega_{\mathbf A}^{*}(t)$ and $\omega_{f}^{*}(t)$ are nonnegative functions determined by $\omega_{\mathbf A}(t)$ and $\omega_{f}(t)$, respectively, as well as  $\mu$ and the prescribed data;\,  $\omega_0^{*}(t)$ and $\omega_1^{*}(t)$ are nonnegative functions determined by $\mu$ and the prescribed data.
Moreover, all the function $\omega_{\bullet}^{*}(t)$ in \eqref{eq0906m} satisfy
\[
\lim_{t \to 0+} \omega_{\bullet}^{*}(t)=0.
\]

Once the estimates \eqref{eq0858m} and \eqref{eq0906m} are available, we can drop the assumption that $u \in C^2(\overline \Omega)$ by the usual bootstrap and approximation argument.
We break the proof of the estimates into several steps.

Unlike Dirichlet or conormal derivative boundary condition cases, we could not find a global $W^{2,p}$ estimate suitable to us in the existing literature.
For this reason, we provide a proof which dispenses with a global $L^p$ estimates, which also works for other boundary conditions.

\subsection*{Step 1}
We first establish interior estimates.
Let us rewrite the equation as
\[
a^{ij}D_{ij} u = f_1:=f - b^i D_i u - cu.
\]
For $B(x_0, 4r_0) \subset \Omega$, the proof of \cite[Theorem~1.10]{DK17} with scaling (c.f. \cite[(2,17)]{DK17}) yields the estimate
\begin{equation}					\label{eq1054m}
\norm{D^2 u}_{L^\infty(B(x_0, 2r_0))} \le C r_0^{-n} \norm{D^2 u}_{L^1(B(x_0,3r_0))} + C \int_0^{r_0} \frac{\tilde \omega_{f_1}(t)}{t}\,dt.
\end{equation}
Here, we adopted an abuse of notation
\[
\tilde \omega_{f_1}(t) = \sum_{i=1}^\infty \kappa^{i \mu} \left( \omega_{f_1}(\kappa^{-i}t)[\kappa^{-i}t \le r_0] + \omega_{f_1}(r_0)[\kappa^{-i}t > r_0] \right),
\]
where $\mu \in (0,1)$ is an arbitrary number, $\kappa = \kappa(n, \lambda, \Lambda, \mu) \in (0,\frac12)$ is a constant, and
\[
\omega_{f_1}(t) := \sup_{x\in B(x_0, 3r_0)} \omega_{f_1; x}(t),\quad\text{where}\quad \omega_{f; x}(t):= \fint_{B(x,t)} \Abs{f-\bar f_{B(x,t)}}.
\]
It should be noted that $\tilde \omega_{f_1}$ satisfies the Dini condition provided $\omega_{f_1}$ satisfies the Dini condition  (see \cite[Lemma~1]{Dong12}).
In particular, if $\omega_{f_1}(t) \lesssim t^a$ with $0<a < \mu$, then $\tilde \omega_{f_1}(t) \lesssim t^a$ as well.

By the proof of Lemma~\ref{lem1721th}, for $B(x,t) \subset \Omega$, we have
\begin{multline}						\label{eq1607th}
\omega_{f_1; x}(t) \le \omega_{f; x}(t) + C(n) \left( [u]_{1; B(x,t)} \,\omega_{\vec b;x}(t) +  t^\mu [u]_{1, \mu; B(x,t)} \,\norm{\vec b}_{L^\infty(B(x,t))} \right) \\
+[u]_{0; B(x,t)} \,\omega_{c;x}(t) +  t^\mu [u]_{0,\mu; B(x,t)} \,\norm{c}_{L^\infty(B(x,t))}.
\end{multline}
Then, by the estimate \eqref{eq1054m}, we obtain
\begin{align}								\nonumber
[u]_{2; B(x_0, 2r_0)} &\le C r_0^{-n} \norm{D^2 u}_{L^1(B(x_0, 3r_0))} + C \int_0^{r_0} \frac{\tilde \omega_{f}(t)}{t}\,dt \\	
										\nonumber
&+C \left( [u]_{1; B(x_0, 4r_0)} \int_0^{r_0} \frac{\tilde \omega_{\vec b}(t)}{t}\,dt+ r_0^\mu [u]_{1, \mu; B(x_0, 4r_0)}\, \norm{\vec b}_{L^\infty(B(x_0, 4r_0))} \right) \\	
										\label{eq1142w}
&+C \left( [u]_{0; B(x_0, 4r_0)} \int_0^{r_0} \frac{\tilde \omega_{c}(t)}{t}\,dt+ r_0^\mu [u]_{0,\mu; B(x_0, 4r_0)}\, \norm{c}_{L^\infty(B(x_0, 4r_0))} \right).
\end{align}
Recall the interpolation inequalities
\begin{equation}					\label{eq1430tu}
[u]_{k;B_r} +[u]_{k, \mu;B_r}\le  [u]_{2; B_r}+C_r \norm{u}_{L^1(B_r)}\quad (k=0,1).
\end{equation}
where $C_r$ is a constant depending on $r$ (and $n$, $k$, and $\mu$).
By setting
\[
\theta(r):= \int_0^{r} \frac{\tilde \omega_{\vec b}(t)}{t}\,dt + \int_0^{r} \frac{\tilde \omega_{c}(t)}{t}\,dt + r^\mu\norm{\vec b}_{L^\infty(\Omega)} + r^\mu \norm{c}_{L^\infty(\Omega)}
\]
and applying \eqref{eq1430tu} to \eqref{eq1142w}, we obtain
\[
[u]_{2; B(x_0, 2r_0)} \le C \theta(r_0) [u]_{2; \Omega} +C \norm{u}_{L^1(\Omega)}
+ C \norm{D^2 u}_{L^1(\Omega)} + C \int_0^{r_0}\frac{\tilde \omega_{f}(t)}{t}\,dt.
\]
Therefore, by choosing $r_0$ small, for $\Omega':=\set{x \in \Omega: \dist(x, \partial \Omega)\ge 4r_0}$, we have
\begin{equation}						\label{eq1451w}
[u]_{2; \Omega'} \le  \frac12\,  [u]_{2; \Omega} +C \norm{u}_{L^1(\Omega)}+ C \norm{D^2 u}_{L^1(\Omega)}+ C \int_0^{r_0} \frac{\tilde \omega_{f}(t)}{t}\,dt,
\end{equation}
where $C$ is a constant depending on the prescribed data and $r_0$.

\subsection*{Step 2}
We turn to estimates near the boundary by closely following the idea of Safonov \cite{Safonov}.
In this step, we shall temporarily assume that $\beta^0 \equiv 0$.
First, we modify \cite[Theorem~2.1]{Safonov} to $C^{2, \rm{Dini}}$ setting (see Lemma \ref{lem:a1} in the Appendix), so that via a local $C^{2,\rm{Dini}}$ diffeomorphism, the boundary condition becomes
\[
\vec \beta(x) \cdot \nabla u(x) = D_n \hat u(y) = \hat g(y),
\]
where $\hat g(y)=g(x)$ is still of $C^{1,\rm{Dini}^2}$.
Moreover, $\hat u$ satisfies the equation
\[
\hat a^{ij} D_{ij} \hat u+  \hat b^i D_i \hat u + \hat c  \hat u=\hat f,
\]
where the coefficients $\hat a^{ij}(y)$, $\hat b^i(y)$, $\hat c(y)$, and the data $\hat f(y)$ are of Dini mean oscillation by Lemma~\ref{lem1721th}.
Preserving the same notation for the transformed objects, we see that the proof is reduced to the case
\[
D_n u = g \quad\text{on}\quad \partial\Omega \cap B(x_0, r_1),\quad r_1 = \text{const.} >0.
\]

A slight modification of \cite[Theorem~2.2]{Safonov} (see Lemma~\ref{lem:a2} in the Appendix), gives us $r>0$ and $v \in C^{2,\rm{Dini}}(\overline \Omega)$ satisfying
\[
D_n v = g \quad\text{on}\quad \partial\Omega \cap B(x_0, r)
\]
with its $C^{2,\rm{Dini}}$ characteristic determined by $g$ and other prescribed data.
Setting $u_0=u-v$, we have
\[
\mathscr{L} u_0=f- \mathscr{L} v=f_0\;\text{ in }\;\Omega,\quad D_n u_0=0\;\text{ on }\partial\Omega \cap B(x_0, r).
\]
By Lemma~\ref{lem:a2}, we also have
\[
\abs{v}_{2; \Omega} \le C \abs{g}_{1; \Omega}+ C \int_0^{r} \frac{\varrho_{Dg}(ct)}{t} + C \int_0^{r} \frac{\varrho_{D\gamma}(c t)}{t}\,dt
\]
and for $x$, $y \in \Omega$, we have
\[
\abs{D^2 v(x) - D^2v(y)} \le C\int_0^{\abs{x-y}} \frac{\varrho_{Dg}(ct)}{t}\,dt + C \int_0^{\abs{x-y}} \frac{\varrho_{D\gamma}(ct)}{t}\,dt.
\]
By Lemma~\ref{lem1721th}, we see that $f_0$ is of Dini mean oscillation in $\Omega$ and
\[
\omega_{f_0}(t) \le \omega_f(t)+ C\int_0^t \frac{\varrho_{Dg}(cs)}{s}\,ds+ C \left( \abs{g}_{1;\Omega} +\int_{0}^{r} \frac{\varrho_{Dg}(ct)}{t}\,dt \right)\omega_0(t) +C \omega_1(t),
\]
where $C$, $c>0$ are constants depending on the prescribed data, and $\omega_0(t)$, $\omega_1(t)$ are nonnegative functions determined by the prescribed data satisfying the Dini condition.
As a matter of fact, we have
\begin{align*}
\omega_{0}(t)& =\omega_{\mathbf A}(t)+ \omega_{\vec b}(t)+ t \norm{\vec b}_\infty+ \omega_c(t)+ t\norm{c}_\infty, \\
\omega_{1}(t) &= \int_{0}^{r} \frac{\varrho_{D\gamma}(ct)}{t}\,dt \cdot(\omega_{\mathbf A}(t) + t\norm{\vec b}_\infty)+  \int_0^t\frac{\varrho_{g}(cs)+\varrho_{D\gamma}(cs)}{s}\,ds,
\end{align*}
where $\gamma$ is a $C^{1, \rm{Dini}^2}$ function that represents $\partial\Omega \cap B(x_0, r)$; see Remark~\ref{rmk1043f} below.
Therefore, in light of \eqref{eq0858m} and \eqref{eq0906m}, by considering $u-v$ instead of $u$, it remains to prove the theorem under the assumption
\begin{equation}				\label{eq1801w}
D_n u = 0 \quad\text{on}\quad \partial\Omega \cap B(x_0, r),\quad r = \text{const.} >0.
\end{equation}

Next, we flatten the boundary by using a ``regularized distance'' function $\psi$ described in Lemma~\ref{lem:a5} in Appendix, which is originally introduced by Lieberman \cite{Lieb85}.
We note that $D\psi \neq 0$ near $\partial \Omega \cap B(x_0, r)$.
Therefore, $C^{1,\rm{Dini}}$ diffeomorphism
\[
x \in \Omega_{2s}(x_0) \longleftrightarrow z=z(x) \in \tilde \Omega_{2s}:=  z(\Omega(x_0, 2s)),
\]
where
\begin{equation}				\label{eq1746sat}
z^i=x^i-x_0^i\quad (i=1,\ldots, n-1),\quad z^n=\psi(x),
\end{equation}
is well defined for some $s \in (0, r/2]$.
For $x \in \Omega(x_0, 2s)$, $z=z(x)$, let us define $\tilde u(z)=u(x)$.
Then, we have
\[
D_i u(x)= D_k \tilde u(z) \,D_i z^k(x),
\]
\begin{equation}					\label{eq1127th}
D_{ij} u(x)= D_{kl} \tilde u(z)\, D_i z^k(x) \, D_j z^l(x) + h^{ij}(x),
\end{equation}
where
\begin{equation}					\label{eq1204th}
h^{ij}(x)= D_n \tilde u(z)\, D_{ij} \psi(x).
\end{equation}
Therefore, the equation is turned into
\[
\tilde{a}^{ij} D_{ij} \tilde u+  \tilde b^i D_i \tilde u + \tilde c \tilde u=\tilde f \quad\text{in}\quad \tilde \Omega_{2s}= z(\Omega(x_0, 2s)) \subset \bR^n_{+}.
\]
where
\begin{gather*}
\tilde a^{ij}=\tilde a^{ij}(z)=a^{kl}(x) D_k z^i(x) D_l z^j(x),\quad \tilde b^i =\tilde b^i(z)=b^k(x)D_k z^i(x),\\
\tilde c=\tilde c(z)=c(x),\quad \tilde f=\tilde f(z)=f(x)-a^{kl}(x) h^{kl}(x),
\end{gather*}
and the boundary condition \eqref{eq1801w} yields
\begin{equation}					\label{eq1120th}
D_n \tilde u=0 \quad \text{on}\quad z(\partial \Omega \cap B(x_0,2s)) \subset \partial \bR^n_{+}=\set{z^n =0}.
\end{equation}

Now, let us choose $s_0 \simeq s$ such that $B^{+}(0, 4s_0) \subset \tilde \Omega_{2s}$.
Since $z=z(x)$ is of $C^{1,\rm{Dini}}$, we see from Lemma~\ref{lem1721th} that $\tilde a^{ij}$, $\tilde b^i$ and $\tilde c$ are of Dini mean oscillation in $B^{+}(0, 4s_0)$.
Moreover, the next lemma shows that $\tilde{h}^{ij}(z)=h^{ij}(x)$ are Dini continuous in $B^{+}(0, 4s_0)$.
\begin{lemma}				\label{lem4.17sat}
Denote $B^+_{4s_0}=B^+(0, 4s_0)$ and let $\tilde h^{ij}(z)=h^{ij}(x)$.
We have
\begin{align*}
\abs{\tilde h^{ij}(z)} &\le C \norm{D^2 \tilde{u}}_{L^{\infty}(B^+_{4s_0})}\,\vartheta(z^n),\\
\abs{\tilde h^{ij}(z_1)-\tilde h^{ij}(z_2)}&\le C \norm{D^2\tilde{u}}_{L^{\infty}(B^+_{4s_0})}\,\vartheta(\abs{z_1-z_2}),
\end{align*}
where $C$ is a constant and $\vartheta(t)=\varrho_{D \psi_0}(t)$ is a nonnegative function satisfying the Dini condition; see Lemma~\ref{lem:a5}.
\end{lemma}

\begin{proof}
By Lemma~\ref{lem1822sat}, we may assume that $\vartheta(t)/t$ is decreasing for $t \in (0, 4s_0)$.
For any $z \in B^+$, by \eqref{eq1120th} and the mean value theorem, we get
\[
\abs{D_n \tilde u(z)}= \abs{D_n \tilde u(z)-D_n \tilde u(\bar z)} \le z^n\norm{D^2 \tilde u}_{L^{\infty}(B^+_{4s_0})},\quad \text{where }\;\bar z = (z^1,\ldots, z^{n-1},0).
\]
Then, by \eqref{eq1204th}, \eqref{eq1746sat}, and Lemma~\ref{lem:a5}, we have
\[
\abs{\tilde h^{ij}(z)} =\abs{h^{ij}(x)} \le C\norm{D^2 \tilde u}_{L^{\infty}(B^+_{4s_0})}\,\vartheta(z^n),\quad \vartheta(t)=\varrho_{D\psi_0}(t).
\]
Also, since
\[
D_k \tilde{h}^{ij}(z)=D_{nm}\tilde u(z) D_k z^m(x) D_{ij}\psi(x) + D_n \tilde u(z) D_{ijk}\psi(x),
\]
 we also get
\begin{equation}				\label{eq2020sat}
\abs{D \tilde{h}^{ij}(z)} \le C \norm{D^2 \tilde u}_{L^{\infty}(B^+_{4s_0})}\,\vartheta(z^n)/z^n.
\end{equation}
Consider any two points $z_1$, $z_2\in B^+_{4s_0}$ with $z_2^n \ge z_1^n$.
In the case when $\abs{z_2-z_1} > \frac{1}{2} z_2^n$, we have
\begin{align*}
\abs{\tilde h^{ij}(z_2)-\tilde h^{ij}(z_1)}&\le \abs{\tilde h^{ij}(z_2)}+ \abs{\tilde h^{ij}(z_1)} \le C\norm{D^2 \tilde u}_{L^{\infty}(B^+_{4s_0})}\, \vartheta(z_2^n)+C\norm{D^2 \tilde u}_{L^{\infty}(B^+_{4s_0})}\, \vartheta(z_1^n) \\
&\le C\norm{D^2 \tilde u}_{L^{\infty}(B^+_{4s_0})}\vartheta(\abs{z_2-z_1}),
\end{align*}
where we used $\vartheta(at)\gtrsim \vartheta(t)$ for $a\ge 1/2$.
On the other hand, in the case when $\abs{z_2-z_1} \le \frac{1}{2}z_2^n$, we have
\[
z_2^n \le \abs{z_1-z_2} + z_1^n \le \tfrac{1}{2} z_2^n + z_1^n,
\]
and thus, we have $\abs{z_1-z_2} \le \frac{1}{2} z_2^n \le z_1^n$.
By the mean value theorem, there is $z_3$ in the line segment $[z_1,z_2]$ satisfying
\[
\abs{\tilde h^{ij}(z_2)-\tilde h^{ij}(z_1)} \le \abs{D\tilde h^{ij}(z_3)}\, \abs{z_2-z_1}.
\]
Note that we have $\abs{z_1-z_2} \le z_1^n \le z_3^n$.
Hence, by using \eqref{eq2020sat}, we obtain
\begin{align*}
\abs{\tilde h^{ij}(z_2)- \tilde h^{ij}(z_1)} &\le \abs{D\tilde h^{ij}(z_3)} \,\abs{z_2-z_1} \le C\norm{D^2 \tilde u}_{L^{\infty}(B^+_{4s_0})}\frac{\vartheta(z_3^n)}{z_3^n}\, \abs{z_2-z_1} \\
& 
\le C\norm{D^2 \tilde u}_{L^{\infty}(B^+_{4s_0})} \vartheta(\abs{z_1-z_2}),
\end{align*}
where we used that $\vartheta(t)/t$ is decreasing.
This completes the proof.
\end{proof}
By Lemmas~\ref{lem4.17sat} and \ref{lem1721th}, we find that $\tilde f=\tilde f(z)=f(x)-a^{kl}(x) h^{kl}(x)$ is of Dini mean oscillation in $B^+=B^+(0,4s_0)$ and there is a constant $a>0$ such that
\begin{equation}					\label{eq1547th}
\omega_{\tilde f}(t) \le C\left(\omega_f(a t)+ \omega_{\mathbf A}(at) [\tilde u]_{2; B^+_{4s_0}} + \vartheta(at) [\tilde u]_{2; B^+_{4s_0}}\right),\quad 0<\forall t < 4s_0.
\end{equation}
Now we set
\[
\tilde f_1:= \tilde f- \tilde b^i D_i\tilde u - \tilde c \tilde u.
\]
Note that by Lemma~\ref{lem1721th}, we have (c.f. \eqref{eq1607th} and \eqref{eq1547th} above)
\begin{multline*}
\omega_{\tilde f_1}(t)  \le \omega_{\tilde f}(t)+ C \left( [\tilde u]_{1; B^+_{4s_0}} \,\omega_{\tilde{\vec b}}(t) +  t^\mu [\tilde u]_{1, \mu; B^+_{4s_0}} \,\norm{\tilde{\vec b}}_{L^\infty(B^+_{4s_0})} \right)\\
+C \left([\tilde u]_{0; B^+_{4s_0}} \,\omega_{\tilde c}(t) +  t^\mu [\tilde u]_{0,\mu; B^+_{4s_0}} \,\norm{\tilde c}_{L^\infty(B^+_{4s_0})} \right), \quad 0<\forall t<4s_0.
\end{multline*}
Also, by the interpolation inequalities (c.f. \eqref{eq1430tu} above) we have
\[
[\tilde u]_{0; B^+_{4s_0}} + [\tilde u]_{0, \mu; B^+_{4s_0}}+[\tilde u]_{1; B^+_{4s_0}} + [\tilde u]_{1, \mu; B^+_{4s_0}} \le   [\tilde u]_{2; B^+_{4s_0}}+C \norm{\tilde u}_{L^1(B^+_{4s_0})}.
\]
Then by \eqref{eq1547th}, for any $0<t<4s_0$, we have
\begin{equation}				\label{eq1559th}
\omega_{\tilde f_1}(t)  \le C \omega_f(at) + C \bigl(\vartheta_0(t) + \vartheta_1(t) \bigr) [\tilde u]_{2; B^+_{4s_0}} + C \vartheta_1(t) \norm{\tilde u}_{L^1(B^+_{4s_0})}.
\end{equation}
where we set
\begin{align*}
\vartheta_0(t)&:= \omega_{\mathbf A}(at)+ \vartheta(at),\\
\vartheta_1(t)&:= \omega_{\tilde{\vec b}}(t)+ \omega_{\tilde c}(t)+ \norm{\tilde{\vec b}}_{L^\infty(B^+_{4s_0})}\, t^\mu + \norm{\tilde c}_{L^\infty(B^+_{4s_0})}\, t^\mu.
\end{align*}
Note that $\vartheta_0(t)$ and $\vartheta_1(t)$ both satisfy the Dini condition.

Therefore, we are reduced to
\[
\tilde a^{ij} D_{ij} \tilde u = \tilde f_1 \;\text{ in }\; B^{+}(0, 4s_0), \quad D_n \tilde u=0 \;\text{ on }\; T(0,4s_0),
\]
where $\tilde f_1$ is of Dini mean oscillation.
By Proposition~\ref{prop4.1} and \eqref{eq1559th}, we have
\begin{multline}					\label{eq1837w}
[\tilde u]_{2; B^+_{s_0}} \le C \norm{D^2 \tilde u}_{L^1(B^+_{4s_0})} +  C\left( \int_0^{s_0} \frac{\hat \vartheta_0(t)}{t} dt + \int_0^{s_0} \frac{\hat \vartheta_1(t)}{t} dt \right) [\tilde u]_{2; B^+_{4s_0}} \\
+C \int_0^{s_0} \frac{\hat \omega_f(at)}{t} dt+
C \left( \int_0^{s_0} \frac{\hat \vartheta_1(t)}{t} dt \right) \norm{\tilde u}_{L^1(B^+_{4s_0})}.
\end{multline}

Note that the equalities \eqref{eq1127th} and Lemma~\ref{lem4.17sat} imply
\[
[u]_{2; \Omega(x_0, \delta s_0)} \le C [\tilde u]_{2; B^{+}_{s_0}}
\]
for some constant $0<\delta < \frac14$.
We also have
\[
[\tilde u]_{2; B^+_{4s_0}} \le C [u]_{2; \Omega},
\]
because the mapping $x=x(z)$ has the same properties as $z=z(x)$.

By requiring $s_0$ so small that we have
\[
C\left(\int_0^{s_0} \frac{\hat \vartheta_0(t)}{t} dt + \int_0^{s_0} \frac{\hat \vartheta_1(t)}{t} dt \right)\le \frac12.
\]
Therefore, we get from \eqref{eq1837w} that
\begin{equation}					\label{eq1842w}
[u]_{2; \Omega(x_0, \delta s_0)} \le C \norm{u}_{W^{2,1}(\Omega)} +  \frac12 [u]_{2; \Omega} +C \int_0^{s_0} \frac{\hat \omega_f(at)}{t} dt+
C \left( \int_0^{s_0} \frac{\hat \vartheta_1(t)}{t} dt \right) \norm{u}_{L^1(\Omega)}.
\end{equation}
By combining \eqref{eq1451w} and \eqref{eq1842w}, we get \eqref{eq0858m}.
Then, \eqref{eq0906m} is obtained by combining \eqref{eq1131m} and the interior estimate appears in the proof of \cite[Theorem 1.6]{DK17}.

\subsection*{Step 3}
Finally, we drop the temporary assumption that $\beta^0 \equiv 0$.
We rewrite the boundary condition as
\[
\vec \beta \cdot \nabla u= g_1:= g - \beta^0 u\quad\text{on}\quad \partial \Omega.
\]
Recall Definition~\ref{def:dini} and observe that
\[
\varrho_{D(\beta^0 u)}(t) \le \varrho_{D \beta^0}(t) [u]_{0; \Omega}+ [\beta^0]_{1; \Omega} [u]_{0,\mu; \Omega}\, t^\mu + [\beta^0]_{1; \Omega} [u]_{1; \Omega}\,t+ [\beta^0]_{0; \Omega} [u]_{1,\mu; \Omega}\, t^\mu.
\]
Therefore, by the interpolation inequalities
\[
[u]_{0; \Omega} + [u]_{0, \mu; \Omega}+[u]_{1; \Omega} + [u]_{1, \mu; \Omega} \le \varepsilon  [u]_{2; \Omega}+C_{\varepsilon} \norm{u}_{L^1(\Omega)},
\]
we find that $\beta^0 u \in C^{1,\rm{Dini}^2}(\overline \Omega)$ and its $C^{1,\rm{Dini}^2}$ characteristic is determined by that $\beta^0$ and the right-hand side of the above inequality.
By choosing $\varepsilon$ small, we can hide $[u]_{2; \Omega}$ contribution.
This completes the proof of Theorem~\ref{thm-main-oblique}.\qed

\subsection*{Proof of Proposition~\ref{prop4.1}}
Once again, we derive an a priori estimate of the modulus of continuity of $D^2u$ by assuming that $u$ is in $C^2(\overline B{}^+_3)$.
As before, we fix some $p \in (0,1)$ and introduce
\[
\phi(x,r) :=\inf_{\mathbf q \in \mathbb S(n)} \left(\fint_{B(x,r)\cap B_4^+} \abs{D^2u-\mathbf q}^p \right)^{\frac1p},
\]
where $\mathbb S(n)$ is the set of all $n\times n$ symmetric real matrices.

We shall derive an estimate for $\phi(\bar x, r)$ for $\bar x \in T(0,3)$ and $0< r \le 1$.
We split $u=v+w$, where $w \in W^{2,2}(B^{+}(\bar x, r))$ is a strong solution of the mixed  problem
\begin{gather*}
\bar a^{ij}  D_{ij} w = - \tr((\mathbf{A} - \bar{\mathbf A}) D^2 u) +f - \bar f \;\text{ in } \; B^{+}(\bar x, r),\\
u =0 \;\text{ on }\; \partial B(\bar x, r)\cap \bR^n_{+},\quad D_n u=0 \;\text{ on }\; T(\bar x, r),
\end{gather*}
where $\bar{\mathbf A} = \bar{\mathbf A}_{B^{+}(\bar x, r)}$ and $\bar f = \bar f_{B^{+}(\bar x, r)}$.
By Lemma~\ref{lem2.6} with scaling, we see that
\[
\Abs{\set{x \in B^{+}(\bar x, r) : \abs{D^2w(x)} >t} } \le \frac{C}{t} \left( \int_{B^{+}(\bar x, r)} \abs{f-\bar f} + [u]_{2; B^{+}(\bar x, r)} \int_{B^{+}(\bar x, r)} \abs{\mathbf A - \bar{\mathbf A}} \right).
\]
Then similar to \cite[(2.11)]{DK17}, we get
\begin{equation}					\label{eq1725th}
\left(\fint_{B^{+}(\bar x, r)} \abs{D^2 w}^p\right)^{\frac1p} \le C [u]_{2; B^{+}(\bar x, r)} \omega_{\mathbf{A}}(r) +C\omega_f(r).
\end{equation}
Next $v:=u-w$  solves
\[
\bar{a}^{ij} D_{ij}v= \bar f \;\mbox{ in } \; B^{+}(\bar x, r), \quad D_n v=0 \;\mbox{ on }\; T(\bar{x},r).
\]
Hence, for any $k, l \in \set{1, \ldots, n-1}$ and $c \in \bR$, the function $V:= D_{kl}v -c$ satisfies
\[
\bar a^{ij} D_{ij} V= 0 \;\mbox{ in } \; B^{+}(\bar x, r), \quad D_n V=0 \;\mbox{ on }\; T(\bar x,r).
\]
By applying Lemma~\ref{lem2.7} with scaling, we see that
\[
\norm{D D_{kl} v}_{L^\infty(B^{+}(\bar x, \frac12 r))} \le C r^{-1}  \left( \fint_{B^{+}(\bar x, r)}\, \abs{D_{kl}v - c}^p\right)^{\frac{1}{p}}.
\]
Therefore, by setting
\[
D_{x'}^2v:=  \set{D_{ij}v : 1\le i,j \le n-1},
\]
we find that
\begin{equation}					\label{eq1759th}
\norm{D D_{x'}^2 v}_{L^\infty(B^{+}(\bar x, \frac12 r))} \le C r^{-1}  \left( \fint_{B^{+}(\bar x, r)}\, \abs{D^2 v - \mathbf q}^p\right)^{\frac{1}{p}},\quad \forall  \mathbf q \in \mathbb S(n).
\end{equation}
Since
\[
D_{nn} v=-\frac{1}{\bar{a}^{nn}}\sum_{(i, j) \neq (n,n)}\bar{a}^{ij}D_{ij}v+\bar f,
\]
by taking the partial derivative with respect to $x^m$, we obtain
\begin{equation}					\label{eq1819th}
D_{nnm} v=-\frac{1}{\bar{a}^{nn}}\sum_{(i, j) \neq (n,n)}\bar{a}^{ij}D_{ijm}v,
\end{equation}
and thus it follows from \eqref{eq1759th} that for any $m \in \set{1,\ldots, n-1}$, we have
\[
\norm{D^2 D_m v}_{L^{\infty}(B^{+}(\bar x, \frac12 r))} \le Cr^{-1}\left( \fint_{B^{+}(\bar x, r)} \abs{D^2 v-\mathbf q}^p \right)^{\frac{1}{p}},\quad \forall  \mathbf q \in \mathbb S(n).
\]
Then, by taking $m=n$ in \eqref{eq1819th} we get
\[
\norm{D^3v}_{L^{\infty}(B^{+}(\bar x, \frac12 r))}  \le Cr^{-1}\left( \fint_{B^{+}(\bar x, r)} \abs{D^2v-\mathbf q}^p \right)^{\frac1p}, \quad \forall  \mathbf q \in \mathbb S(n).
\]
Let $0 < \kappa \le \frac12$ be a constant to be fixed later.
By the mean value theorem, we have
\[
\left( \fint_{B^{+}(\bar x, \kappa r)} \abs{D^2v-(D^2v)_{\kappa r}}^p \right)^{\frac1p} \le 2\kappa r \norm{D^3v}_{L^{\infty}(B^{+}(\bar x, \frac12 r))},\quad\text{where }\; (D^2 v)_{\kappa r}:=\fint_{B^{+}(\bar x, \kappa r)} D^2 v.
\]
Hence, we see that there is some constant $C_0=C_0(n, \lambda, \Lambda)$ such that
\[
\left( \fint_{B^{+}(\bar x, \kappa r)} \abs{D^2v-(D^2v)_{\kappa r}}^p \right)^{\frac1p} \le C_0\kappa \left( \fint_{B^{+}(\bar x, r)} \abs{D^2v-\mathbf q}^p \right)^{\frac1p}, \quad \forall  \mathbf q \in \mathbb S(n).
\]
By using the decomposition $u = v+w,$ similar to \eqref{eq0750f}, we obtain
\begin{align*}
&\left( \fint_{B^+(\bar x, \kappa r)}\Abs{D^2u-(D^2v)_{\kappa r}}^p \right)^{\frac1p}\\
&\qquad \qquad  \le 2^{\frac{1-p}{p}}\left( \fint_{B^+(\bar x, \kappa r)} \Abs{D^2v-(D^2v)_{\kappa r}}^p \right)^{\frac1p}+2^{\frac{1-p}{p}} \left( \fint_{B^+(\bar x, \kappa r)} \abs{D^2w}^p \right)^{\frac1p}\\
& \qquad \qquad \le 4^{\frac{1-p}{p}}C_0 \kappa \left( \fint_{B^{+}(\bar x, r)} \abs{D^2u-\mathbf q}^p \right)^{\frac1p}+C(\kappa^{-\frac{n}p}+1)\left( \fint_{B^{+}(\bar x, r)} \abs{D^2w}^p \right)^{\frac1p}.
\end{align*}
Since $\mathbf q \in \mathbb S(n)$ is arbitrary, by using \eqref{eq1725th}, we obtain
\[
\phi(\bar{x},\kappa r) \le 4^{\frac{1-p}{p}}C_0 \kappa\,\phi(\bar x, r) + C(\kappa^{-\frac{n}p}+1)\left( \omega_{\mathbf{A}}(r)\norm{Du}_{L^{\infty}(B^{+}(\bar x, r))}+\omega_{f}(r) \right),
\]
which is analogous to \eqref{eq0750f}.
Also, we note that \eqref{eq1209f} is available whenever $B(x, \rho) \subset B^{+}(\bar x, R)$ and $\rho \simeq R$.

Consequently, we have Lemmas 2.16, 2.17, 2.18, and 2.19 in \cite{DEK17}  available in our setting with $\phi(\bar x, r)$.
In particular, by \cite[Lemma~2.18]{DEK17}, we obtain \eqref{eq0830f} .
Also, the estimate \eqref{eq1131m} follows by a similar argument employed in deriving \cite[(2.36)]{DEK17}.
\qed

\begin{remark}					\label{rmk1043f}
Observe that by Lemmas~\ref{lem1721th} and \ref{lem:a2}, we have
\begin{align*}
\omega_{\mathbf{A} D^2 v}(t)  & \lesssim [v]_{2;\Omega}\,  \omega_{\mathbf A}(t)+ \norm{\mathbf A}_\infty \,\varrho_{D^2 v}(t) \\
& \lesssim \left(\abs{g}_{1; \Omega}+\int_{0}^{2b} \frac{\varrho_{Dg}(ct)}{t} + \frac{\varrho_{D\gamma}(c t)}{t}\,dt\right)\omega_{\mathbf A}(t)+  \int_0^t \frac{\varrho_{Dg}(cs)}{s} \,ds+ \int_0^t\frac{\varrho_{D\gamma}(cs)}{s}\,ds,\\
\omega_{\vec b D v}(t) & \lesssim [v]_{1; \Omega} \,\omega_{\vec b}(t) + \norm{\vec b}_\infty \,\varrho_{Dv}(t) \lesssim \abs{g}_{1; \Omega}\, \omega_{\vec b}(t) + t \norm{\vec b}_\infty\, [v]_{2; \Omega} \\
&\lesssim \abs{g}_{1; \Omega}\, \omega_{\vec b}(t)+ t \left(\abs{g}_{1; \Omega}+\int_{0}^{2b} \frac{\varrho_{Dg}(ct)}{t} + \frac{\varrho_{D\gamma}(c t)}{t}\,dt\right) \norm{\vec b}_\infty,\\
\omega_{c v}(t)  & \lesssim [v]_{0; \Omega}\, \omega_c(t) + \norm{c}_\infty \,\varrho_v(t) \lesssim \abs{g}_{1; \Omega}\, \omega_{c}(t) + t \norm{c}_\infty\, [v]_{1; \Omega} \\
&\lesssim \abs{g}_{1; \Omega}\, \omega_{c}(t) + t \abs{g}_{1; \Omega}\, \norm{c}_\infty.
\end{align*}
Therefore, we have
\begin{align*}
\omega_{f_0}(t) &\le \omega_f(t)+ \omega_{\mathscr{L} v}(t)  \lesssim \omega_f(t)+ \omega_{\mathbf{A} D^2 v}(t)+ \omega_{\vec bDv}(t)+\omega_{cv}(t) \\
& \lesssim\omega_f(t)+ \left( \abs{g}_{1; \Omega} +\int_{0}^{2b} \frac{\varrho_{Dg}(ct)}{t}\,dt \right) \cdot \left( \omega_{\mathbf A}(t)+ \omega_{\vec b}(t)+ t \norm{\vec b}_\infty+ \omega_c(t)+ t\norm{c}_\infty  \right) \\
&\quad+ \int_{0}^{2b} \frac{\varrho_{D\gamma}(ct)}{t}\,dt \cdot \left(\omega_{\mathbf A}(t)+ t\norm{\vec b}_\infty \right) + \int_0^t \frac{\varrho_{Dg}(cs)}{s} \,ds+ \int_0^t\frac{\varrho_{D\gamma}(cs)}{s}\,ds.
\end{align*}
\end{remark}

\section{Appendix}
In the Appendix, we provide the proofs for some technical lemmas used before by slightly modifying those in Safonov's paper \cite{Safonov}.

\begin{lemma}			\label{lem:a5}
Let $\Omega$ be a $ C^{1,\rm{Dini}}$ domain with a defining $C^{1,\rm{Dini}}$ function $\psi_0$ (see Remark \ref{rem2.6}).
Then there exists a function $\psi \in C^{1,\rm{Dini}}(\bR^n) \cap  C^\infty(\Omega)$ such that
\[
\delta \psi(x) \le d_x:=\dist(x, \partial \Omega) \le \delta^{-1} \psi(x),\quad \forall x \in \Omega,
\]
\[
[\psi]_{1; \Omega} \le 1,\quad \varrho_{D\psi}(t) \le C\varrho_{D\psi_0}(ct),
\]
where $[\cdot]_{k; \Omega}$ and $\varrho_{\bullet}$ are as defined in  \eqref{eq0853m} and Definition~\ref{def:dini}, respectively, $ \delta=\delta(n,\psi_0)\in (0,1)$, $C=C(n, \psi_0)$, and $c=c(n, \psi_0)>0$.
Moreover, for any multi-index $l$ with $\abs{l}=m\ge 2$, we have
\[
\abs{D^l \psi(x)} \le C \psi(x)^{1-m} \varrho_{D\psi_0} (\psi(x)),\quad \forall x \in \Omega,
\]
where $C=C(n, m, \psi_0)$.
\end{lemma}

\begin{proof}
We modify the proof of \cite[Lemma~2.4]{Safonov}.
Since $\psi_0$ is Lipschitz and $\abs{D\psi_0}\ge 1$ on $\partial\Omega$, there exist constants $K>0$ and $\delta \in (0,1)$, depending on $n$ and $\psi_0$, such that
\[
\abs{\psi_0(x)-\psi_0(y)} \le K \,\abs{x-y},\quad  \forall x, y \in \bR^n
\]
and
\begin{equation}				\label{eq1037w}
\delta \psi_0(x) \le \dist(x, \partial \Omega) \le \delta^{-1} \psi_0(x),\quad \forall x \in \Omega.
\end{equation}
Let us consider a function
\begin{equation}				\label{eq1127w}
\Psi(t,x)=\psi_0^{(t)}(x)=\int \psi_0(x-ty)\zeta(y)dy\quad\text{on}\quad \bR^{n+1},
\end{equation}
where $\zeta$ is a standard mollifier, that is a $C^\infty$ function supported in the unit ball $B(0,1)$ satisfying $0\le \zeta \le 1$ and $\int \zeta =1$.
Then it follows
\[
\abs{\Psi(t_1,x)-\Psi(t_2, x)} \le K \abs{t_1-t_2}.
\]
Therefore, we can define the implicit function
\begin{equation}			\label{eq2115m}
\psi(x)=t=(2K)^{-1} \Psi(t,x)\quad\text{on}\quad \bR^n.
\end{equation}
We have
\[
\abs{2K \psi(x)-\psi_0(x)} = \abs{\Psi(t,x)-\Psi(0,x)} \le K \abs{t}=K \abs{\psi(x)}.
\]
This inequality implies
\[
\frac13 \le K\frac{\psi(x)}{\psi_0(x)} \le 1\quad\text{on}\quad \bR^n \setminus \partial \Omega.
\]
Hence, by \eqref{eq1037w}, we have
\[
\delta K t=\delta K \psi(x) \le \dist(x, \partial \Omega)\le 3K \delta^{-1} \psi(x)=3K\delta^{-1}t,\quad\forall x \in \Omega.
\]

Now, it follows from \eqref{eq1127w} that $\Psi \in C^{1, \rm{Dini}}(\bR^{n+1})$ and
\begin{equation}			\label{eq1208sat}
\sup_{\bR^{n+1}} \,\abs{D \Psi} \le \sup_{\bR^n}\, \abs{D \psi_0} \le K.
\end{equation}
Also, since
\[
D_t \Psi(t_1, x_1)- D_t \Psi(t_2, x_2)= \int_{B(0,1)} \left(D\psi_0(x_1-t_1y)-D\psi_0(x_2-t_2 y)\right)\cdot(-y) \zeta(y)\,dy
\]
and
\[
D_{x^i}\Psi(t_1, x_1)- D_{x^i} \Psi(t_2, x_2) = \int_{B(0,1)} \left(D_i\psi_0(x_1-t_1y)-D_i\psi_0(x_2-t_2 y)\right) \zeta(y)\,dy,
\]
we have
\begin{equation}			\label{eq1209sat}
\varrho_{D \Psi}(r) \le C \varrho_{D \psi_0}(cr),
\end{equation}
where $C=C(n)$ and $c=c(n)>0$.
Therefore, we have $\psi \in C^{1,\rm{Dini}}(\bR^n)$.
Moreover, by \eqref{eq2115m}, we find that
\[
D_i \psi(x)=(2K)^{-1} D_t \Psi(\psi(x), x) D_i \psi(x)+ (2K)^{-1} D_{x^i} \Psi(\psi(x), x),
\]
and thus by \eqref{eq1208sat}, we obtain
\begin{equation}			\label{eq1210sat}
\sup_{\bR^n} \,\abs{D \psi} \le 1.
\end{equation}
Moreover, since
\begin{align*}
\abs{D_i \psi(x)-D_i \psi(y)} &\le (2K)^{-1} \abs{D_t \Psi(\psi(x), x)} \,\abs{D_i \psi(x) - D_i \psi(y)} \\
&\quad+ (2K)^{-1} \abs{D_t \Psi(\psi(x),x)- D_t \Psi(\psi(y),y)} \, \abs{D_i \psi_0(y)}\\
&\quad +(2K)^{-1} \abs{D_i \Psi(\psi(x),x)- D_i \Psi(\psi(y),y)},
\end{align*}
we get
\[
\abs{D_i \psi(x)-D_i \psi(y)} \le \frac12 \abs{D_i \psi(x)-D_i \psi(y)} + \left(\frac12 +\frac{1}{2K} \right) \varrho_{D\Psi}(a\abs{x-y}),\quad a= \sqrt{1+K^2},
\]
and thus by \eqref{eq1209sat}
\begin{equation}				\label{eq1914eve}
\varrho_{D\psi}(\tau) \le C \varrho_{D\psi_0}(c \tau),\quad\text{where }\; C=C(n, K)\;\text{and }\; c=c(K).
\end{equation}
Furthermore, for any multi-index $l \in \bZ_{+}^{n+1}$ with $\abs{l}=m \ge 2$, $t \neq 0$, using \eqref{eq1127w}, we obtain similar to Lemma~2.3 of \cite{Safonov} that
\begin{equation}			\label{eq2114m}
\abs{D^l \Psi(t,x)} \le C(n,m)\, t^{1-m} \varrho_{D\psi_0}(t).
\end{equation}
Indeed, we have
\[
D_t \Psi(t,x)= \int D\psi_0(x-ty)\cdot (-y) \zeta(y)\,dy= -\abs{t}^{-n} \int D_k\psi_0(z) \,\zeta_k\left(\frac{x-z}{t}\right)\,dz,
\]
where we set $\zeta_k(x):=x^k\zeta(x)$, and
\[
D_{x^i} \Psi(t,x)=\int D_i\psi_0(x-ty) \zeta(y)\,dy=\abs{t}^{-n} \int D_i\psi_0(z) \,\zeta\left(\frac{x-z}{t}\right)\,dz.
\]
Therefore, for $t>0$, we have
\[
D_{tt} \Psi(t,x) =n t^{-n-1} \int D_k\psi_0(z) \zeta_k \left(\frac{x-z}{t}\right)\,dz + t^{-n-1} \int D_k \psi_0(z)\, \tilde \zeta_k\left(\frac{x-z}{t}\right)\,dz,
\]
where we set $\tilde\zeta_k(x):=x\cdot D\zeta(x)$.
Since $\int \zeta_k=0$ and $\int \tilde\zeta_k =-n \int \zeta_k=0$, we have
\begin{multline*}
D_{tt} \Psi(t,x)= n t^{-n-1} \int \left(D_k \psi_0(z)-D_k \psi_0(x)\right) \zeta_k \left(\frac{x-z}{t}\right)\,dz\\
+t^{-n-1} \int \left(D_k \psi_0(z)-D_k \psi_0(x)\right) \tilde\zeta_k \left(\frac{x-z}{t}\right)\,dz.
\end{multline*}
Since the above integrals are actually taken over $B(x,t)$, we have
\[
\abs{D_{tt} \Psi(t,x)} \le C(n) t^{-n-1} \varrho_{D\psi_0}(t) \,\abs{B(x, t)} \left( \norm{\zeta_k}_\infty+ \norm{\tilde \zeta_k}_\infty\right) \le C(n) t^{-1} \varrho_{D\psi_0}(t).
\]
By a similar computation, we get
\[
\abs{D_{tx^i} \Psi(t,x)}  \le C(n) t^{-1} \varrho_{D\psi_0}(t),\quad \abs{D_{x^i x^j} \Psi(t,x)} \le C(n) t^{-1} \varrho_{D\psi_0}(t).
\]
We have thus shown \eqref{eq2114m} for $m=2$ and $t>0$.
The general cases can be deduced in the same fashion.

For a multi-index $l \in \bZ_{+}^n$ with $\abs{l}=m \ge 2$, by the chain rule and a direct computation, we obtain from \eqref{eq2115m}, \eqref{eq1208sat}, \eqref{eq1210sat}, and \eqref{eq2114m} that
\[
\abs{D^l \psi(x)} \le C (\psi(x))^{1-m} \varrho_{D\psi_0} (\psi(x)),\quad \forall x \in \Omega,
\]
The lemma is proved.
\end{proof}

\begin{corollary}			\label{cor:a5}
Assume the same hypothesis as in Lemma~\ref{lem:a5}.
Then, for any function $u \in C^{1,\rm{Dini}}(\overline \Omega)$, there exists a function $\tilde u \in C^{1,\rm{Dini}}(\overline \Omega) \cap C^\infty(\Omega)$ such that $\tilde u=u$ on $\partial \Omega$,
\[
\abs{\tilde u}_{1; \Omega} \le C\abs{u}_{1; \Omega}\quad\text{and}\quad \varrho_{D\tilde u}(t) \le C \left(\norm{Du}_{L^\infty(\Omega)}\, \varrho_{D\psi_0}(ct)+\varrho_{Du}(ct) \right),
\]
where $\abs{\cdot}_{k; \Omega}$ and $\varrho_{\bullet}$ are as defined in  \eqref{eq0854m} and Definition~\ref{def:dini}, respectively, $C=C(n, \psi_0)$, and $c=c(n, \psi_0)>0$.
Moreover, for any multi-index $l$ with $\abs{l}=m\ge 2$, we have
\[
\abs{D^l \tilde u(x)} \le C d_x^{1-m} \left(\varrho_{Du}(d_x)+ \varrho_{D\psi_0}(cd_x) \right),\quad \forall x \in \Omega,\quad d_x:=\dist(x,\partial \Omega),
\]
where $C=C(n, m, \psi_0)$ and $c=c(\psi_0)>0$.
\end{corollary}

\begin{proof}
We modify the proof of \cite[Corollary~2.1]{Safonov}.
Let $\psi$ be from Lemma~\ref{lem:a5}.
Similar to \eqref{eq1127w}, define
\[
U(t,x)=u^{(t)}(x)=\int u(x-ty)\zeta(y)dy\quad\text{on}\quad \bR^{n+1}.
\]
Then, the function
\begin{equation}				\label{eq1744sun}
\tilde u(x)=U(\delta \psi(x),x)
\end{equation}
is well defined in $\overline \Omega$, and $\tilde u=u$ on $\partial \Omega$.
It is clear that
\[
[\tilde u]_{0; \Omega} \le [U]_{0; \bR \times \Omega} \le [u]_{0; \Omega},
\]
where $[\cdot]_{k; \Omega}$ is as defined in \eqref{eq0853m}.
Moreover, since
\[
D_{i} \tilde u(x)= \delta D_tU(\delta \psi(x),x)  D_i \psi(x) + D_i U(\delta \psi(x),x),
\]
we have
\begin{equation}				\label{eq1225sat}
[\tilde u]_{1; \Omega} \le C [U]_{1; \bR\times \Omega} \le C [u]_{1; \Omega},
\end{equation}
where $C=C(n, \psi_0)$.
As in \eqref{eq1209sat}, we also have
\begin{equation}				\label{eq1226sat}
\varrho_{DU}(t) \le C \varrho_{Du}(ct),
\end{equation}
where $C=C(n)$ and $c=c(n)$.
Furthermore, since
\begin{multline*}
\abs{D_i \tilde u(x)-D_i \tilde u(y)} \le \delta \abs{D_t U(\delta \psi(x), x)} \,\abs{D_i \psi(x) - D_i \psi(y)} \\
+ \delta \abs{D_t U(\delta\psi(x),x)- D_t U(\delta\psi(y),y)} \, \abs{D_i \psi(y)} + \abs{D_i U(\delta \psi(x),x)- D_i U(\delta \psi(y),y)},
\end{multline*}
we have (recall $\norm{D\psi}_\infty \le 1$)
\[
\varrho_{D \tilde u}(\tau) \le C \norm{DU}_{\infty} \,\varrho_{D\psi}(\tau) + C\varrho_{DU}(a\tau),
\]
where $C=C(n, \psi_0)$ and $a=\sqrt{1+\delta^2}$.
Then, by \eqref{eq1225sat}, \eqref{eq1226sat}, and \eqref{eq1914eve}, we have
\[
\varrho_{D\tilde u}(\tau) \le C \norm{Du}_{\infty}\, \varrho_{D\psi_0}(c\tau)+C \varrho_{Du}(c\tau),
\]
where $C=C(n, \psi_0)$ and $c=c(n,\psi_0)>0$.

Finally, similar to \eqref{eq2114m}, for any multi-index $l \in \bZ_{+}^{n+1}$ with $\abs{l}=m \ge 2$, we get
\[
\abs{D^l U(t,x)} \le C(n,m)\, t^{1-m} \varrho_{Du}(t),\quad t> 0.
\]
Also, by \cite[(2.20)]{Safonov}, for any multi-index $l \in \bZ_{+}^{n+1}$ with $\abs{l}= 1$, we have
\[
\abs{D^l U(t,x)} \le C(n) \norm{Du}_\infty,\quad t >0.
\]
Then, by using the above two inequalities, for any multi-index $l \in \bZ_{+}^n$ with $\abs{l}=m \ge 2$, we derive from \eqref{eq1744sun} that
\[
\abs{D^l \tilde u(x)} \le C \psi(x)^{1-m} \left( \varrho_{Du}(\delta \psi(x))+ \norm{Du}_\infty \, \varrho_{D\psi_0}(\psi(x))\right).\qedhere
\]
\end{proof}

\begin{lemma}				\label{lem:a2.5}
Let $\tau>0$, $n_0 \in \bN,$ and let a $n_0 \times n_0$ matrix function $\mathbf{A}(t)=[A^{ij}(t)]$ and a vector valued function $\vec B(t)$ with values in $\bR^{n_0}$ be defined and continuous on $[0,\tau)$.
Suppose that
\begin{equation}			\label{eq:a2.37}
\abs{\mathbf A(t)} \le K_0, \quad \abs{\vec B(t)} \le K_1e^{K_0t}(\tau - t)^{-2} \varrho(\tau-t)
\end{equation}
on $[0, \tau)$ for some constants $K_0$, $K_1 \ge 0$, and a function $\varrho$ on $[0,\tau)$ satisfying $\varrho(t)>0$  and $\frac{d}{dt}(t^{-\mu}\varrho(t)) \le 0$ for some $\mu \in (0,1)$.
Then every solution $\vec X(t)$ of the system
\[
\frac{d\vec X}{dt} = \mathbf{A} \vec X + \vec B, \quad 0 \le t \le \tau,
\]
satisfies the estimate
\begin{equation}			\label{eq:a2.39}
\abs{\vec X(t)} \le N_0 e^{K_0t}(\tau - t)^{-1} \varrho(\tau - t),\quad  0 \le t < \tau,
\end{equation}
where $N_0 = \max \Set{\tau \varrho(\tau)^{-1}\abs{\vec X(0)}, K_1/(1-\mu)}$.
\end{lemma}

\begin{proof}
We modify the proof of \cite[Lemma~2.5]{Safonov}.
Consider the function
\begin{equation}			\label{eq:a2.40}
f(t) = e^{-2K_0t}(\tau - t)^2 \abs{\vec X(t)}^2/ \varrho(\tau-t)^2,\quad 0\le t <\tau.
\end{equation}
Obviously the inequality \eqref{eq:a2.39} is equivalent to $f(t) \le N_0^2$.
By the choice of $N_0$, we have $f(0) \le N_0^2.$

Suppose that \eqref{eq:a2.39} fails for some $t \in (0, \tau)$.
Then there exist $\varepsilon>0$ and $t_0 \in (0, \tau)$ such that
\begin{equation}			\label{eq:a2.41}
f(t) < N_0^2+\varepsilon \; \mbox{ on }\; [0,t_0), \quad f(t_0)= N_0^2+\varepsilon.
\end{equation}
Moreover, since $K_1 \le (1-\mu)N_0$, by \eqref{eq:a2.37} we have
\[
\abs{\vec B(t_0)} \le (1-\mu) N_0 e^{K_0 t_0}(\tau-t_0)^{-2}\varrho(\tau-t_0) < (1-\mu) (\tau-t_0)^{-1} \abs{\vec X(t_0)}.
\]
Therefore, for $t=t_0$,
\begin{equation}				\label{eq1152w}
\frac{d}{dt}\abs{\vec X}^2 = 2\vec X \cdot \frac{d\vec X}{dt} = 2\vec X \cdot (\mathbf{A} \vec X + \vec B) <2\left(K_0+\frac{1-\mu}{\tau - t_0}\right)\abs{\vec X}^2.
\end{equation}
Also, by the assumption that $\frac{d}{dt}(t^{-\mu}\varrho(t)) \le 0$, we obtain
\begin{equation}				\label{eq1153w}
\varrho'(t) \le \mu t^{-1}\varrho(t).
\end{equation}
By differentiating the equation \eqref{eq:a2.40} and using the above two inequalities, we get
\begin{multline*}
f'(t) = -2K_0e^{-2K_0t}(\tau - t)^2\abs{\vec X(t)}^2/\varrho(\tau-t)^2 - e^{-2K_0t}2(\tau - t)\abs{\vec X(t)}^2/ \varrho(\tau-t)^2 \\
+ e^{-2K_0t}(\tau - t)^2 \frac{d}{dt}\abs{\vec X(t)}^2 /\varrho(\tau-t)^2+2e^{-2K_0t}(\tau - t)^2\abs{\vec X(t)}^2 \varrho'(\tau-t)/\varrho(\tau-t)^3.
\end{multline*}
Then, by using inequalities \eqref{eq1152w} and \eqref{eq1153w} we get
\[
f'(t_0) < \frac{e^{-2K_0t_0}(\tau - t_0)\abs{\vec X(t_0)}^2}{\varrho(\tau-t_0)^2} \left( - 2K_0(\tau - t_0)-2+ 2K_0(\tau-t_0)+ 2(1-\mu) + 2\mu \right)=0.
\]
On the other hand, \eqref{eq:a2.41} yields $f'(t_0) \ge 0$.
This contradiction proves the estimate \eqref{eq:a2.39}.
\end{proof}

\begin{lemma}					\label{lem:a1}
Let $\Omega\subset \bR^n$ be a $C^{1, \rm{Dini}^2}$ domain with a defining function $\psi_0$ and $\vec \beta=(\beta^1,\ldots, \beta^n) \in C^{1,\rm{Dini}^2}(\overline \Omega)$ satisfy the condition \eqref{oblique}.
Then there exists a constant $r>0$, and for every $x_0\in\partial\Omega,$ there exists a one-to-one $C^{2,\rm{Dini}}$ mapping $\vec \Phi: B(x_0,r) \to \bR^n$ such that upon writing $y=\vec \Phi(x)$
and $x=\vec \Psi(y)$, we have
\[
\beta^i=\frac{\partial x^i}{\partial y_n}=\frac{\partial \Psi^i}{\partial y_n} \quad\text{on}\quad \partial\Omega \cap B(x_0,r),\quad i=1,\ldots,n.
\]
The $C^{2,\rm{Dini}}$ characteristics of $\vec \Phi$ and $\vec \Psi$ are determined only by the given data, namely, $\mu_0$ and $C^{1,\rm{Dini}^2}$ characteristics of $\partial\Omega$ and $\vec \beta$.
\end{lemma}
\begin{proof}
We slightly modify the proof of \cite[Theorem~2.1]{Safonov} using Corollary~\ref{cor:a5} and Lemma~\ref{lem:a2.5} instead of \cite[Corollary~2.1]{Safonov} and \cite[Lemma~2.5]{Safonov}, respectively.

We follow exactly the same proof of Theorem~2.1 in \cite{Safonov} up to the beginning of the evaluation of second and third derivatives of $x=x(y)$.
In particular, we use the same symbolic notation there so that
\begin{equation}				\label{eq2.49saf}
\frac{d}{dt} \frac{\partial x^i}{\partial y^j}= \sum_k \frac{\partial \beta^i}{\partial x^k} \frac{\partial x^k}{\partial y^j},\quad \text{i.e., }\; \frac{d}{dt}\frac{\partial x}{\partial y} = \frac{\partial \vec \beta}{\partial x} \frac{\partial x}{\partial y}
\end{equation}
turns into the form $dx'/dt=\sum \vec \beta' x'$.
Then, we have the estimates $\abs{\vec \beta'} \le N$, $\abs{x'} \le N$ with different constants $N>0$.
Differentiating \eqref{eq2.49saf} twice, we obtain that $x''=\partial^2 x/\partial y_i \partial y_j$ and $x'''=\partial^3 x/\partial y^i \partial y^j \partial y^k$ satisfy the systems
\begin{equation}				\label{eq2.53saf}
\frac{d}{dt} x''= \sum \vec \beta' x'' + \sum \vec \beta'' x' x',
\end{equation}
\[
\frac{d}{dt} x'''= \sum \vec \beta' x''' + \sum \vec \beta'' x'' x'+ \sum \vec \beta''' x' x' x',
\]
which correspond to \cite[(2.53)]{Safonov} and \cite[(2.54)]{Safonov} .
Here, by using Corollary~\ref{cor:a5}, we modified $\vec \beta$ in such a manner that $\vec \beta \in C^{1, \rm{Dini}^2}(\overline \Omega(x_0, r)) \cap C^\infty(\Omega(x_0,r))$, where $\Omega(x_0, r)= \Omega \cap B(x_0, r)$ for some $r>0$, and for $\abs{l}=m \ge 2$ and $x \in \Omega(x_0,r)$,
\begin{equation}			\label{eq22.26w}
\abs{D^l \vec \beta(x)} \le N d_x^{1-m} \left( \varrho_{D\vec \beta}(d_x)+\varrho_{D \psi_0}(d_x)\right),\quad d_x=\dist(x, \partial \Omega).
\end{equation}
Then by \eqref{eq22.26w} and \cite[(2.47)]{Safonov}, we get (by replacing $\varrho_\bullet(t)=\varrho_\bullet(Nt)$ if necessary)
\begin{equation}				\label{eq2.55saf}
\begin{aligned}
\abs{\vec \beta''(x)} &\le N (\tau-t)^{-1}\left(\varrho_{D\vec\beta}(\tau -t)+ \varrho_{D\psi_0}(\tau -t)\right), \\
\abs{\vec \beta'''(x)} &\le N (\tau-t)^{-2}\left(\varrho_{D\vec\beta}(\tau -t)+ \varrho_{D\psi_0}(\tau -t)\right).
\end{aligned}
\end{equation}
Let us apply Lemma~\ref{lem:a2.5} to the system \eqref{eq2.53saf}, where
\[
\textstyle \vec X=\set{x''},\quad {\mathbf A}\vec X =\set{\sum \vec \beta' x''},\quad \vec B=\set{\sum \vec \beta'' x'x'},
\]
to get (note that $(\tau -t)^{-1} \le N (\tau-t)^{-2}$)
\begin{equation}				\label{eq23.38w}
\abs{x''} \le N(\tau-t)^{-1}\left(\varrho_{D\vec \beta}(\tau -t) + \varrho_{D \psi_0}(\tau -t) \right) \le N (\tau-t)^{-1}.
\end{equation}
Now, let us apply Lemma~\ref{lem:a2.5} to \eqref{eq23.38w}, where
\[
\textstyle\vec X=\set{x'''},\quad {\mathbf A} \vec X =\set{\sum \vec \beta' x'''},\quad \vec B=\set{\sum \vec \beta'' x'' x'+ \sum \vec \beta''' x'x'x'}.
\]
The estimate \eqref{eq2.55saf}, \eqref{eq23.38w} provide us
\[
\abs{\vec B(t)} \le N (\tau-t)^{-2}\left(\varrho_{D\vec\beta}(\tau -t)+ \varrho_{D\psi_0}(\tau -t)\right),
\]
hence (by replacing $\varrho_\bullet(t)$ with $\varrho_\bullet(c t)$ if necessary)
\begin{equation}				\label{eq2.57w}
\abs{x'''} \le N(\tau -t)^{-1}\left(\varrho_{D\vec\beta}(\tau -t)+ \varrho_{D\psi_0}(\tau -t)\right) \le N d_x^{-1} \left(\varrho_{D\vec\beta}(d_x)+ \varrho_{D\psi_0}(d_x)\right).
\end{equation}
Since $x=x(y)$ is the $C^1$ diffeomorphism, the inverse mapping $y=y(x) \in C^1(\overline \Omega_r)$, and
\[
\hat d_y=\dist(y, \partial \hat \Omega_r) \le N d_x\quad \text{for }\;x \in \Omega_r,\; y=y(x)\in \hat \Omega_r=y(\Omega_r).
\]
Therefore, from \eqref{eq2.57w} it follows
\begin{equation}
                            \label{eq10.18}
\abs{x'''(y)} \le N \hat d_y^{-1} \left(\varrho_{D\vec\beta}(\hat d_y)+ \varrho_{D\psi_0}(\hat d_y)\right).
\end{equation}

Finally, we estimate the modulus of continuity of $x''$ by modifying the proof of \cite[Lemma~2.1]{Safonov}.
Let us fix $y_1$, $y_2 \in \hat \Omega_r$, and set $r=\abs{y_1-y_2}$.
One can choose $y_0 \in \hat \Omega_r$ such that
\[
B(y_0, r/N) \in \hat \Omega_r,\quad \abs{y_k-y_0} \le Nr\;\text{ for }\; k = 1, 2,
\]
for some $N>0$.
Furthermore, we can connect $y_k$ with $y_0$ by means of a smooth path in $\hat \Omega_r$,
\[
\set{y=h_k(s): 0\le s \le s_k},\quad s_k\le Nr,\quad h_k(0)=y_k,\quad h_k(s_k)=y_0,
\]
parameterized by the arc length $s$ in such a manner that
\[
s/N \le  \hat d_{h_k(s)}  \le N,\quad 0\le s \le s_k.
\]
By the mean value theorem and \eqref{eq10.18}, we get
\[
\abs{x''(y_0)-x''(y_k)} \le  \int_0^{s_k} \abs{x'''(h_k(s))}\,ds
\le C  \int_0^{s_k} \frac{\varrho_{D\vec\beta}(\hat d_{h_k(s)})}{\hat d_{h_k(s)}} + \frac{\varrho_{D\psi_0}(\hat d_{h_k(s)})}{\hat d_{h_k(s)}} \,ds.
\]
Again, by Lemma~\ref{lem1822sat}, we may assume without loss of generality that the functions  $\varrho_{D\vec \beta}(t)/t$ and $\varrho_{D\psi_0}(t)/t$ are decreasing.
Then, we have
\begin{align*}
\abs{x''(y_1)-x''(y_2)}& \le \sum_{k=1}^2 \,\abs{x''(y_0)-x''(y_k)}\\
& \le C N \sum_{k=1}^2  \int_0^{s_k} \frac{\varrho_{D\vec\beta}( s/N)}{s} + \frac{\varrho_{D\psi_0}(s/N)}{s} \,ds\\
&\le \tilde C \int_0^r \frac{\varrho_{Dg}(s)}{s}+\frac{\varrho_{D\psi_0}(s)}{s}\,ds,
\end{align*}
where $\tilde C$ is a constant depending only on $n$ and $\gamma$.
\end{proof}

\begin{lemma}					\label{lem:a2}
Let $\Omega\subset \bR^n$ be a $C^{1, \rm{Dini}^2}$ domain with a defining function $\psi_0$, and $\gamma: \bR^{n-1} \to \bR$ be in Definition \ref{c1dini}.
Fix a small $b>0$ so that $\abs{D\gamma(x')}<\frac12$ for any $x'\in \bR^{n-1}$ with $\abs{x'}<b$.
Denote
\[
\rU_b=\set{x=(x',x^n) \in \bR^n:  \gamma(x')< x^{n}<b,\; \abs{x'}<b}.
\]
Then for any function $g \in C^{1,\rm{Dini}^2}(\overline \Omega)$, there exists a function $v \in C^{2,\rm{Dini}}(\overline{\rU}_b)$ such that
\[
D_n v=g \quad\text{on}\quad \partial \Omega.
\]
Moreover, we have
\[
\abs{v}_{1; \rU_b} \le C \abs{g}_{1; \Omega}, \quad [v]_{2; \rU_b} \le  C \abs{g}_{1; \Omega}+C\int_{0}^{2b} \frac{\varrho_{Dg}(ct)}{t} + \frac{\varrho_{D\gamma}(c t)}{t}\,dt,
\]
and for a multi-index $l$ with $\abs{l}=m \ge 3$, we have
\[
\abs{D^l v(x)} \le C  d_x^{2-m} \left( \varrho_{Dg}(c d_x) + \varrho_{D\gamma} (c d_x)\right),\quad \forall x \in \rU_b, \quad d_x=\dist(x, \partial \Omega).
\]
Furthermore, we have
\[
\varrho_{D^2 v}(t) \le C \int_0^t \frac{\varrho_{Dg}(cs)}{s}+\frac{\varrho_{D\gamma}(cs)}{s}\,ds.
\]
In the above,  $C$ and $c$ are constants, which vary from line to line, depending only on $n$, $\gamma$, and $b$.
\end{lemma}

\begin{proof}
We modify the proof of Theorem~2.2 in \cite{Safonov}.
For $x=(x', x^n) \in \rU_b$, we set
\[
\bar d_x=x^n-\gamma(x').
\]
Note that we have $\bar d_x \simeq d_x=\dist(x, \partial \Omega)$ for $x \in \rU_b$.
By Corollary~\ref{cor:a5}, there exists a function $\tilde g \in C^{1, \rm{Dini}^2}(\overline \Omega) \cap C^\infty(\Omega)$ such that $\tilde g=g$ on $\partial \Omega$ and
\begin{equation}				\label{eq1816xmas}
\abs{\tilde g}_{1; \Omega} \le C \abs{g}_{1; \Omega},
\end{equation}
and for any multi-index $l$ with $\abs{l}=m\ge 2$
\begin{equation}				\label{eq1947xmas}
\abs{D^l \tilde g(x)} \le C\bar d_x^{1-m} \left(\varrho_{Dg}(c_1 \bar d_x)+ \varrho_{D\psi_0}(c_2 \bar d_x) \right),\quad \forall x \in \rU_b.
\end{equation}
Now, we define
\[
v(x)=v(x',x^n)= -\int_{x^n}^{b} \tilde g(x',t)\,dt,\quad x \in \rU_b.
\]
Then, it is clear that $D_n v=\tilde g=g$ on $\partial \Omega$.
Moreover, by \eqref{eq1816xmas}, for $x \in \rU_b$, we have
\[
\abs{v(x)} \le 2b [\tilde g]_{0; \Omega} \le C \abs{g}_{1; \Omega} ,\quad
\abs{D_i v(x)} \le 2b [\tilde g]_{1; \Omega} \le C \abs{g}_{1; \Omega},
\]
for $i=1,\ldots, n-1$, and also
\[
\abs{D_n v(x)} = \abs{\tilde g(x)} \le C \abs{g}_{1; \Omega}.
\]
Now, for $l=(l',l^n)$, $\abs{l}=\abs{l'}+ l^n=m\ge 2$, and $x \in \rU_b$, we consider separately the cases $l^n \ge 1$ and $l^n=0$.
If $l^n\ge 1$, then
\[
D^l v(x) = D^{l'} D_n^{l^n}v(x)=D^{l'} D_n^{l^n-1} \tilde g(x),
\]
and thus, when $m=2$, we have
\[
\abs{D^l v(x)} \le [\tilde g]_{1; \Omega} \le C \abs{g}_{1; \Omega},
\]
and when $m \ge 3$, by \eqref{eq1947xmas}, we have
\[
\abs{D^l v(x)} \le C \bar d_x^{2-m} \left(\varrho_{Dg}(c_1 \bar d_x)+ \varrho_{D\gamma}(c_2 \bar d_x) \right).
\]
If $l^n=0$, by using \eqref{eq1947xmas} and noting $\bar d_{(x',t)} = t-\gamma(x')$, we obtain
\begin{equation}				\label{eq1319m}
\abs{D^l v(x)} \le \int_{x^n}^{b} \abs{D^l \tilde g(x',t)}\,dt \le C  \int_{\bar d_x}^{b-\gamma(x')} \frac{\varrho_{Dg}(c_1t)}{t^{m-1}} + \frac{\varrho_{D\gamma}(c_2 t)}{t^{m-1}}\,dt.
\end{equation}
In particular, when $m=2$, we derive from \eqref{eq1319m} that
\[
\abs{D^l v(x)} \le  C \int_{0}^{2b} \frac{\varrho_{Dg}(c_1t)}{t} + \frac{\varrho_{D\gamma}(c_2 t)}{t}\,dt.
\]
Let us fix $\mu \in (0,1)$.
By Lemma~\ref{lem1822sat}, we may assume without loss of generality that the functions  $\varrho_{Dg}(c_1t)/t^\mu$ and $\varrho_{D\gamma}(c_2t)/t^\mu$ are decreasing on $(0, 2b]$.
Hence, when $m \ge 3$ we get from \eqref{eq1319m} that
\begin{align*}
\abs{D^l v(x)} &\le  C \left( \frac{\varrho_{Dg}(c_1 \bar d_x)}{\bar d_x^\mu} + \frac{\varrho_{D\gamma}(c_2 \bar d_x)}{\bar d_x^\mu}\right) \int_{\bar d_x}^{b-\gamma(x')} t^{\mu+1-m}\,dt \\
&\le C  \left( \frac{\varrho_{Dg}(c_1 \bar d_x)}{\bar d_x^\mu} + \frac{\varrho_{D\gamma}(c_2 \bar d_x)}{\bar d_x^\mu}\right)(m-2-\mu) \bar d_x^{\mu+2-m}.
\end{align*}
Therefore, in conclusion, for $\abs{l}=2$, we have
\[
\abs{D^l v(x)} \le C \abs{g}_{1; \Omega}+ C\int_{0}^{2b} \frac{\varrho_{Dg}(c_1t)}{t} + \frac{\varrho_{D\gamma}(c_2 t)}{t}\,dt,\quad \forall x \in \rU_b,
\]
and for $\abs{l}=m\ge 3$, we have
\begin{equation}					\label{eq1543tu}
\abs{D^l v(x)} \le C  \bar d_x^{2-m} \varrho_{Dg}(c_1 \bar d_x) + C  \bar d_x^{2-m}\varrho_{D\gamma} (c_2 \bar d_x),\quad \forall x \in \rU_b.
\end{equation}
Finally, we estimate the modulus of continuity of $D^2 v$.
Let us fix $x_1$, $x_2 \in \rU_b$, and set $r=\abs{x_1-x_2}$.
One can choose $x_0 \in \rU_b$ such that
\[
B(x_0, r/N) \in \rU_b,\quad \abs{x_k-x_0} \le Nr\;\text{ for }\; k = 1, 2,
\]
for some $N>0$.
Furthermore, we can connect $x_k$ with $x_0$ by means of a smooth path in $\rU_b$,
\[
\set{x=h_k(s): 0\le s \le s_k},\quad s_k\le Nr,\quad h_k(0)=x_k,\quad h_k(s_k)=x_0,
\]
parameterized by the arc length $s$ in such a manner that
\[
s/N \le  \bar d_{h_k(s)}  \le N,\quad 0\le s \le s_k.
\]
By the mean value theorem and \eqref{eq1543tu} with $m=3$, we get
\[
\abs{D_{ij}v(x_0)-D_{ij}v(x_k)} \le  \int_0^{s_k} \abs{D D_{ij} v(h_k(s))}\,ds
\le C  \int_0^{s_k} \frac{\varrho_{Dg}(c_1 \bar d_{h_k(s)})}{\bar d_{h_k(s)}} + \frac{\varrho_{D\gamma}(c_2 \bar d_{h_k(s)})}{\bar d_{h_k(s)}} \,ds.
\]
Again, by Lemma~\ref{lem1822sat}, we may assume without loss of generality that the functions  $\varrho_{Dg}(c_1t)/t$ and $\varrho_{D\gamma}(c_2t)/t$ are decreasing on $(0, 2b]$.
Then, we have
\begin{align*}
\abs{D_{ij}v(x_1)-D_{ij}v(x_2)}& \le \sum_{k=1}^2 \,\abs{D_{ij}v(x_0)-D_{ij}v(x_k)}\\
& \le C N \sum_{k=1}^2  \int_0^{s_k} \frac{\varrho_{Dg}(c_1 s/N)}{s} + \frac{\varrho_{D\gamma}(c_2 s/N)}{s} \,ds\\
&\le \tilde C \int_0^r \frac{\varrho_{Dg}(\tilde c s)}{s}+\frac{\varrho_{D\gamma}(\tilde c s)}{s}\,ds,
\end{align*}
where $\tilde C$, $\tilde c$ are constants that depend only on $n$ and $\gamma$.
\end{proof}

\section*{acknowledgment}
The authors would like to thank the anonymous referee for his careful reading and valuable comments.

\def\cprime{$'$}

\end{document}